\documentclass{amsart}
\usepackage{amssymb}
\usepackage{graphicx}

\def\alp{\alpha}

\def\dd{\displaystyle}

\def\gam{\gamma}
\def\Gam{\Gamma}

\def\mtline#1{\hbox to#1{\hrulefill}}

\def\noi{\noindent}

\def\Ome{\Omega}

\def\te{\text}

\def\what{\widehat}
\def\wtit{\widetilde}

\def\cA{{\mathcal A}}

\def\cG{{\mathcal G}}
\def\cH{{\mathcal H}}

\def\cU{{\mathcal U}}

\def\maA{{\mathcal A}}
\def\maB{{\mathcal B}}

\def\maF{{\mathcal F}}
\def\maG{{\mathcal G}}
\def\maH{{\mathcal H}}

\def\maL{{\mathcal L}}
\def\maM{{\mathcal M}}

\def\maS{{\mathcal S}}

\def\C{{\mathbb C}}
\def\D{{\mathbb D}}

\def\H{{\mathbb H}}

\def\K{{\mathbb K}}

\def\N{{\mathbb N}}

\def\Q{{\mathbb Q}}
\def\R{{\mathbb R}}
\def\S{{\mathbb S}}
\def\T{{\mathbb T}}

\def\Z{{\mathbb Z}}

\usepackage{secteqn}
\usepackage{mathabx}

\usepackage{color}
\definecolor{purple}{cmyk}{.33,1,0,.4}
\definecolor{m}{rgb}{1,0.1,1}
\definecolor{green}{cmyk}{1,0,1,0}
\definecolor{test}{rgb}{1,1,1}
\definecolor{cmyk}{cmyk}{0,1,1,0}

\newtheorem{Equation}{}[section]

\newtheorem{corollary}[Equation]{Corollary}
\newtheorem{definition}[Equation]{Definition}
\newtheorem{example}[Equation]{Example}

\newtheorem{lemma}[Equation]{Lemma}
\newtheorem{proposition}[Equation]{Proposition}

\newtheorem{remark}[Equation]{Remark}

\newtheorem{theorem}[Equation]{Theorem}

\newtheorem{prop}[Equation]{Proposition}

\def\ch{\operatorname{ch}}

\def\codim{\operatorname{codim}}
\def\Coker{\operatorname{Coker}}
\def\coker{\operatorname{coker}}

\def\Diff{\operatorname{Diff}}
\def\dim{\operatorname{dim}}

\def\Dir{\operatorname{\not{\hspace{-0.08cm}}\pa}}

\def\Res{\operatorname{Res}}

\def\End{\operatorname{End}}

\def\oH{\operatorname{H}}

\def\Id{\operatorname{I}}
\def\id{\operatorname{id}}

\def\ind{\operatorname{ind}}
\def\Sat{\operatorname{Sat}}
\def\Ind{\operatorname{Ind}}

\def\Iso{\operatorname{Iso}}
\def\oK{\operatorname{K}}
\def\ker{\operatorname{ker}}
\def\Ker{\operatorname{Ker}}

\def\sign{\operatorname{sign}}
\def\Sign{\operatorname{Sign}}

\def\Spin{\operatorname{Spin}}

\def\Td{\operatorname{Td}}
\def\Tr{\operatorname{Tr}}

\def\tr{\operatorname{tr}}

\def\vol{\operatorname{vol}}

\def\cA{{\mathcal A}}

\def\maA{{\mathcal A}}

\def\alp{\alpha}

\def\gam{\gamma}
\def\Gam{\Gamma}

\def\ep{\epsilon}

\def\Lam{\Lambda}

\def\Ome{\Omega}

\def\dd{\displaystyle}

\def\pa{\partial}

\def\ms{\medskip}

\def\ssm{\smallsetminus}

\def\what{\widehat}
\def\wtit{\widetilde}

\def\te{\wtit e}

\begin{document}



\title[Higher Lefschetz theorem \today]
{The higher fixed point theorem for foliations. \\Applications to rigidity and integrality\
\\ \today\\
}


\author{Moulay Tahar Benameur}
\address{Institut Montpellierain Alexander Grothendieck, UMR 5149 du CNRS, Universit\'e de Montpellier}
\email{moulay.benameur@umontpellier.fr}

\author[James L. Heitsch \today]{James L.  Heitsch}
\address{Mathematics, Statistics, and Computer Science, University of Illinois at Chicago} 
\email{heitsch@uic.edu}

    \dedicatory{Dedicated to Fedor Sukochev on the occasion
    of his sixtieth birthday}

\thanks{Mathematical subject classification (1991). 19L47, 19M05, 19K56.\\
Key words: $C^*$-algebras, K-theory, Lefschetz, foliations.}

\begin{abstract} 

We give applications of the higher Lefschetz theorems for foliations of \cite{BH10}, primarily involving Haefliger cohomology.  These results show that the transverse structures of foliations carry important topological and geometric information.  This is in the spirit of the passage from the Atiyah-Singer index theorem for a single compact manifold to their families index theorem, involving a compact fiber bundle over a compact base.  For foliations, Haefliger cohomology plays the role that the cohomology of the base space plays in the families index theorem.  

We obtain highly useful numerical invariants by paring with closed holonomy invariant currents.  In particular, we prove that the non-triviality of the {{higher}} $\what{A}$ genus of the foliation in Haefliger cohomology can be an obstruction to the existence {of non-trivial  leaf-preserving compact connected group actions}.   We then construct a large collection of examples for which no such actions exist.  Finally, we relate our results to Connes' spectral triples, and prove useful integrality results.

\end{abstract}

\maketitle
\tableofcontents

\ms

\section{Introduction} 

In this paper we continue our program of investigating invariants for foliations which come from the fact that their transverse structures carry important topological and geometric information.  These invariants arise from the extension of (generally classical cohomological) invariants to invariants which involve the Haefliger cohomology of the foliation.  This is in the spirit of the passage from the Atiyah-Singer index theorem for a single compact manifold to their families index theorem, involving a compact fiber bundle over a compact base.  For foliations, Haefliger cohomology plays the role that the cohomology of the base space plays in the families index theorem.  

\ms

This paper is devoted to applications of the higher Lefschetz theorems for foliations of \cite{BH10}.   Let $F$ be a foliation of the closed  Riemannian manifold $(V, g)$, and $h$ a leaf-preserving diffeomorphism of $V$ which generates a compact Lie group $H$ of isometries of $(V, g)$. We assume that the fixed point submanifold $V^h=V^H$  of $h$ is  transverse to the foliation and denote by $TF^h = T(F \cap V^h)$ the induced integrable subbundle of $TV^h$. So $(V^h, F^h)$ is a new closed foliated manifold.   Given an $H$-equivariant leafwise elliptic pseudodifferential complex $(E,d)$ on $(V, F)$, the $H$-equivariant (analytic)  index of $(E, d)$ is a  class $\Ind^H(E, d)$ in the $H$-equivariant $K$-theory group $K^H(C^*(V, F))$ of the Connes $C^*$-algebra $C^*(V, F)$ \cite{C79}.  Note that the group $K^H(C^*(V, F))$ is a module over the representation ring $R(H)$ of $H$.  The Lefschetz class $L(h; E, d)$ of $h$ with respect to $(E, d)$ was introduced  in \cite{B97} as the localiazation of the class $\Ind^H(E, d)$ with respect to the prime ideal associated with $h$. So, $L(h; E, d)$ belongs to the localized $R(H)_h$-module $K^H(C^*(V, F))_h$, where the subscript $_h$ means localization of $R(H)$-modules with respect to the ideal associated with $h$ in $R(H)$.  The main theorem of \cite{B97} was the expression of $L(h;E,d)$ in terms of topological data over the fixed point foliation $(V^h, F^h)$. An easy corollary is that if there exists such $(E, d)$ with $L(h; E,d)\neq 0$ then $V^h\neq \emptyset$. More precisely, the $K$-theory  Lefschetz theorem can be stated as follows.

\ms

{\bf{Theorem \ref{K-Lef} \cite{B97}}}
{\em 
Under the above notations and denoting by $i:TF^h\hookrightarrow TF$  the inclusion map, the following fixed-point formula holds in the localized $R(H)_h$-module $K^H(C^*(V, F))_h$:
$$
L(h; E, d) = \left(\Ind _{(V^h, F^h)}\otimes \id_{R(H)_h}\right) \left( \frac{i^*[\sigma (E, d)]}{\lambda_{-1} (N^h\otimes \C)}\right), 
$$
where $\Ind_{(V^h, F^h)} : K (TF^h) \to K (C^*(V^h, F^h))$ denotes the topological Connes-Skandalis  index morphism  for the foliated manifold $(V^h, F^h)$  \cite{CS84}. 
}
\ms

Note that $i^*[\sigma (E, d)]$ is the restriction of the symbol class $[\sigma (E, d)]\in K_H(TF)$ to $TF^h$, and the fraction  $\frac{i^*[\sigma (E, d)]}{\lambda_{-1} (N^h\otimes \C)}$ is to be understood in the localized module $K_H(TF^h)_h\simeq K(TF^h)\otimes R(H)_h$. So, while the RHS of the $K$-theory Lefschetz formula belongs to $K(C^*(V^h, F^h))\otimes R(H)_h$, it is viewed in $K^H(C^*(V,F))_h$ via a standard Mortia extension morphism associated with the transverse submanifold $V^h$. 

\ms

In order to extract scalar Lefschetz fixed point formulae from Theorem \ref{K-Lef}, we were naturally led  in  \cite{BH10} to use equivariant cyclic cohomology. Indeed, any $H$-equivariant cyclic cocycle over the convolution algebra $C_c^\infty (\maG)$ of the holonomy groupoid $\maG$ induces an $R(H)$-equivariant pairing with $K^H(C_c^\infty (\maG))$ with values in the continuous central functions on $H$. So, whenever this pairing extends to the equivariant $K$-theory of the completion algebra $C^*(V, F)$, one gets well defined higher Lefschetz numbers by localizing this pairing at $h$, and also by evaluating the resulting central function at the same $h$. In this paper, we have used this method to investigate the case of the equivariant cyclic cocycle associated with a  transversely elliptic  Dirac-type operator. More precisely, if we assume for simplicity that the foliation is Riemannian and transversely spin, then according to \cite{Kordyukov}, the transverse spin-Dirac operator with coefficients in any basic hermitian bundle provides an appropriate spectral triple which can be restricted to the fixed point foliation $(V^h, F^h)$. 
Then we obtain  the following integrality theorem. 

\ms

{\bf{Theorem \ref{LefInt} }}
{\em 
Denote by $\phi^{CM}$ the even Connes-Moscovici residue cocycle in the $(b, B)$-bicomplex  associated with a transversely elliptic Dirac-type operator on $(V, F)$  \cite{CM95}. Let $\Ind^{CS}_{V^h, F^h}: K(TF^h) \rightarrow K (C^*(V^h, F^h))$ be the Connes-Skandalis topological longitudinal index morphism for the foliation $(V^h, F^h)$ \cite{CS84}. Then we have 
$$
\left\langle (\Ind^{CS}_{V^h, F^h}\otimes \C)\left(\frac{i^*[\sigma (E,d)] (h)}{\lambda_{-1} (N^h\otimes \C) (h)}\right) \, , \, [\phi^{CM}]\vert_{V^h, F^h}\otimes \id_\C\right\rangle\; \in \; R(H)(h). 
$$
}

\ms

Recall that $\phi^{CM}$ is represented by a finite list of residues of zeta functions and the above result says that some rational combination of such residues belongs to the integral subgroup $R(H)(h)$.These residues are closely related to Dixmier traces and more generally to some singular traces on algebras of pseudodifferential operators.  They are in the spirit of the more general singular traces studied by Sukochev and his collaborators in the semi-finite setting, see for instance \cite{LMSZ23, SZ23}. See also \cite{BF06} where examples of (semi-finite) Dixmier traces on foliations are also given.  Integrality results are highly important since they can lead to new invariants of great significance.  As examples, the results of Chern-Simons and Cheeger-Simons led to differential characters and the Simons characters, while the Bott vanishing theorem led to the secondary characteristic classes for foliations.  
If for instance $h$ is a leaf-preserving involution, then we obtain an integer.  If more generally $h$ has order $p\geq 2$ then we obtain an element of $\Z[e^{2i\pi/p}]$. This is in the spirit of the classical Atiyah-Segal theorem and its corollaries, especially those related with number theory and described in the nice monography \cite{HirzebruchZagier}. 
Notice that if the foliation is only a transversely oriented Riemannian foliation, then one may as well use the transverse signature operator as defined in \cite{Kordyukov} and get a similar integrality statement. For non-riemanian foliations though, the computations become highly non-trivial since one needs to use the transverse hypo-elliptic signature operator on the Connes fibration associated with $(V, F)$, along with the more complicated dimension spectrum as well as more complicated residues. So although the result can be stated as above, it can hardly be exploited in practice, so this general case will be addressed elsewhere.

\ms

A different, but closely related approach to the above cyclic pairing method was investigated in \cite{BH10}. In particular, we proved in \cite{BH10} that the above pairing method  works perfectly well for all foliations, when generating the equivariant cyclic cocycles from the Haefliger homology of the foliation. We then proved that the $K$-theory Lefschetz formula can always be paired with any given  closed holonomy invariant current $C$ to provide a rich collection of topological fixed point  formulae, valid  for all foliations, and given by
$$ 
L_C(h; E,d)=\left< \left[C|_{V^h}\right],  \int_{F^h} \frac{\ch_{\C}(i^{*}[\sigma(E,d)](h))}{\ch_{\C}(\lambda_{-1}(N^{h}
\otimes{\C})(h))} \, \Td(TF^{h} \otimes \C) \right >.
$$
Here  $\Td$ is the Todd genus, $\ch_{\C}$ is the complexified Chern character, and $\dd \int_{F^h}$ is the ``integral over the fiber" from the cohomology of $V^h$ to the Haefliger cohomology of $(V^h,F^h)$.  See Section 2 for the specifics.  When the current $C$ corresponds to a holonomy invariant measure one obtains the measured Lefschetz theorem proven in \cite{HL90}.  Even if such a measure doesn't exist, one can still apply this formula with other interesting holonomy invariant currents. In the case of simple foliations, this higher formula reduces to the fiberwise Lefschetz formula proven in \cite{Bena02}, while in the case of foliated flat bundles, it reduces to the formula proven in \cite{Bena03}.

\ms

Returning to the Riemannian case, the pairings with the Kordyukov equivariant  spectral triples described above  turn out to be closely related with the equivariant pairings with Haefliger homology. More precisely, by applying the machinery of  the Getzler rescaling method \cite{Getzler} to Riemannian foliations and transversely elliptic Dirac-type operators, in \cite{BH18, BH19} we succeeded  in representing the class in the $(b,B)$-bicomplex of the Connes-Moscovici residue cocycle by the expected characteristic basic form. Combining these results with our theorems here, we  deduce the following topological integrality result. 

\noi
{\bf{Theorem \ref{integrality} }}
{\em Under the above assumptions,  the  characteristic number 
$$
\int_{V^h} \dfrac{\ch_\C (i^*[\sigma (E,d)](h))}{\ch_\C (\lambda_{-1}(N^h\otimes \C)(h))} \Td (TF^h\otimes \C) \what{A} (\nu^h) \ch ({\what{E}} \,|\,_{V^h}) \; \text{ belongs to }R(H)(h).
$$
}

Note that this formula  is not expected to remain valid for non-Riemannian foliations, while as explained above, Theorem \ref{LefInt} can certainly be stated for more general foliations.

\ms

We also explore other  important consequences of the topological Haefliger Lefschetz formulae.
In Section \ref{LefEGs} for instance, we consider the four classical geometric operators and show that our results give non trivial extensions of results in \cite{HL90}.  Also, depending on the context, $L_C(h; E,d)$ gives highly useful numerical invariants.  A prime example of this is Theorem \ref{rigidity} below,  the higher foliation rigidity theorem.   Let $\what{A}(F)$ be the $\what{A}$-characteristic class of the tangent bundle of the foliation $TF$.  Our higher topological Lefschetz theorem allows us to generalize Proposition 3.2 of \cite{HL91} by taking into account all closed holonomy invariant currents.  Note that $\what{A}(TF) \in \oH^*(V,\R)$, and that $\dd\int_F$ maps $\oH^*(V,\R)$ to the Haefliger cohomology of $F$.

\ms

\noi
{\bf{Theorem \ref{rigidity} }}
{\em {Suppose that $TF$ is even dimensional and spin, and that there exists a closed holonomy invariant current $C$ such that 
$$
\left<[C], \int_F \what{A}(TF) \right> \neq 0.
$$
Then no compact connected Lie group can act non trivially as a group of isometries of $V$ preserving the leaves of F and their spin structure.}}

\ms
In subsection \ref{rigidguys}, we  construct a large collection of examples not already covered by \cite{HL91} for which Theorem \ref{rigidity} shows that no such actions exist.

\ms

\noindent
{\em Acknowledgements.}  It is a pleasure to thank {{our wives V\'eronique and Lynn for putting up with us.}}

MTB  wishes to thank the french National Research Agency for support via the project ANR-14-CE25-0012-01 (SINGSTAR).

JLH wishes to thank the Simons Foundation for a Mathematics and Physical Sciences-Collaboration Grant for Mathematicians, Award Number 632868.

\section{Review of the higher Lefschetz theorem}\label{Review}

In this section we review  the main results of  \cite{BH10} with a brief overview of the main theorem of \cite{B97}.  In particular,  $F$ is a smooth {{oriented}} dimension $p$ foliation of the smooth compact Riemannian manifold $(V,g)$ with the tangent bundle to $F$ denoted by $TF$ {{and normal bundle $\nu$.}}   $\cG$ is the holonomy groupoid of $F$, which consists of equivalence classes of leafwise paths, where two paths are identified if they start at the same point, end at the same point, and the holonomy germ along them is the same.  Composition of paths makes $\cG$ a groupoid, and its space of units $\cG^0$ consists of the classes of the constant paths, so  $\cG^0 \simeq V$.  Denote  
by $\cG_x$ the elements of $\cG$ which start at the point $x \in V$,  by $\cG^y$ those elements  which end at the point $y \in V$, and  by $\cG^y_x$ the intersection $ \cG_x \cap \cG^y$.  
We have the maps $s,r:\cG \to V$, where $s(\gamma) = x$, if $\gamma  \in \cG_x$, and $r(\gamma) = y$ of $\gamma \in \cG^y$.  The metric $g$ on $V$ induces a canonical metric on $\cG$, and so the splitting $T\cG = TF_s \oplus TF_r \oplus \nu_{\cG}$.  Note that  $r_*(\nu_{\cG,\gamma}) = \nu_{r (\gamma)}$, and $s_*(\nu_{\cG,\gamma}) = \nu_{s (\gamma)}$.
For details, see \cite {BH04}. 
The metric on $\cG$ gives metrics on the  submanifolds $\cG_x, \cG^y \subset \cG$.   So objects such as $L^{2}(\cG_{x})$ and $L^{2}(\cG^y)$ are well defined, and do not depend on the choice of metric since $V$ is compact.  
Note that $r:\cG_{x} \to L_x$ is the holonomy covering of $L_x$, the leaf of $F$ through $x$, and similarly, $s:\cG^y \to L_y$ is  the holonomy covering of $L_y$.  

\ms

The (reduced) Haefliger cohomology of $F$, \cite{Hae}, is given as follows.   Let  ${\cU}$ be a finite good cover of $M$ by foliation charts as defined in \cite{HL90}.  For each $U_i \in {\cU}$, let $T_i\subset U_i$ be a transversal and set $T=\bigcup\,T_i$.  We may assume that the closures of the $T_i$ are disjoint.  Let $\cH$ be the holonomy pseudogroup
induced by $F$ on $T$.   Give $\cA^k_c(T)$, the space of k-forms on $T$ with compact support,  the usual $C^\infty$ topology, and denote the exterior derivative by $d_T:\cA^k_c(T)\to  \cA^{k+1}_c(T)$.   Denote by  $\cA^k_c(M/F)$ quotient of $\cA^k_c(T)$ by the closure of the vector subspace generated by elements of the form $\alpha-h^*\alpha$ where $h\in \cH$ and $\alpha\in\cA^k_c(T)$ has support contained in the range of $h$.  The exterior derivative $d_T$ induces a continuous differential $d_H:\cA^k_c(M/F)\to \cA^{k+1}_c(M/F)$.  Note that $\cA^k_c(M/F)$ and $d_H$ are independent of the choice of cover $\cU$.  The associated cohomology theory is denoted $H^*_c(M/F)$ and is called the Haefliger cohomology of $F$.  

\ms 

A holonomy invariant $k$-current $C$ assigns a real number to any compactly supported differential $k$ form defined on any transversal, with the stipulation that $C(h_\gamma^*\alpha - \alpha)=0$.
Any such $C$ gives a continuous (for the smooth topology) linear form on $\cA_c^k(M/F)$, and such a form is called a Haefliger current.  

\ms

Denote by $\cA^{p+k}(M)$ the space of smooth $p+k$-forms on $M$.  As the bundle $TF$ is oriented, there is a continuous open surjective linear map, called integration over the leaves,
$
\dd\int_F :\cA^{p+k}(M)\longrightarrow \cA^k_c(M/F)
$
which commutes with the exterior derivatives $d_{M}$ and $d_{H}$, so it induces the map 
$$ 
\int_F :\oH^{p+k}(M;\R) \to \oH^k_c(M/F).
$$
This map is given by 
$
\dd \int_F \omega \,\, = \,\, \sum_i \int_{U_i} \phi_i \omega,
$
 where $\{\phi_i\}$ is  a partition of unity subordinate to the cover $\cU$, and  
$\dd \int_{U_i}$ is integration over the fibers of the projection $U_i \to T_i$.  

\ms

Let $H \subset \Iso (V,g)$ be a closed subgroup of the compact Lie group of isometries of $(V, g)$ which acts by  $F$-preserving  transformations.  The isometry of $(V,g)$ which corresponds to the action of an element $h\in H$ will also  be denoted by $h$, so such $h$ takes leaves of $F$ to leaves of $F$.  Connes' $C^*$-algebra of $(V,F)$  is as usual denoted $C^*(V,F)$, see  \cite{C82}.  It is easy to check that $C^*(V,F)$ is an $H$-algebra, i.e.\ the induced action of $H$ is strongly continuous for the $C^*$-norm.  In fact, one can prove as well that the smooth subalgebra $C_c^\infty (\maG)$, or more generally any of its variants $C_c^\infty (\maG, E)$ corresponding to coefficients in a given $H$-equivariant vector bundle $E$ over $V$, is  an $H$-algebra for the smooth compact-open topology. 

\ms

{{Denote by  $(E,d)$ a leafwise elliptic pseudodifferential complex }}on $(V, F)$ which is $H$-equivariant, i.e.\ $E=\oplus E^i$, $d = \{d_i\}$, the bundles $E^i$ are $H$-equivariant, while the operators $d^i$ are $H$-invariant leafwise  pseudodifferential operators acting between the sections of $E^i$ and $E^{i+1}$,  such that $d^{i+1}d^i=0$. Ellipticity of the leafwise complex $(E, d)$ means that the  corresponding pointwise complex of principal symbols is exact when restricted to  the leafwise cosphere bundle. Then the $H$-equivariant index class of $(E, d)$ is a well defined element of the equivariant $K$-theory group $K^H (C^*(V, F))$ and actually factors through the equivariant $K$-theory of the smooth algebra $C_c^\infty (\cG, E)$,  see \cite{BH04} and also \cite{BH10}. Since $K^H (C^*(V, F))$ is an $R(H)$-module, we can localize it with respect to {{any conjugacy class $[h]$ in $H$ and obtain the localized index class of $(E, d)$ at $[h]$. }}

\ms

We proved in \cite{BH10} that the $H$-equivariant $K$-theory group of the smooth algebra $C_c^\infty (\maG; E)$ pairs with its $H$-equivariant cyclic cohomology $H_\lambda (C_c^\infty (\maG; E), H)$ to produce central functions on $H$. More specifically, we proved the following result.

\begin{theorem}\cite{BH10}\label{Pairing}
For any compact group $H$ acting by $F$-preserving isometries of $(V, g)$, there exists  a well defined pairing 
$$
K^H(C_c^\infty (\maG; E))  \otimes H^{even}_\lambda (C_c^\infty (\maG; E), H) \longrightarrow C(H)^H,
$$
where $H^{even}_\lambda (C_c^\infty (\maG; E), H)$ is the even $H$-equivariant cyclic cohomology and $C(H)^H$ denotes the central continuous functions on $H$. 
\end{theorem}

As we explain below, elements of $H^{even}_\lambda (C_c^\infty (\maG; E), H)$ are provided by $H$-invariant transversely elliptic operators as studied in \cite{CM95}, but also by any holonomy invariant closed current  when for instance $H$ is connected.
When $H$ is abelian, localization at a given element $h$ of $H$ is  well defined and the following definition was introduced by the first author in \cite{B97}.  

\begin{definition}\cite{B97}\
Let $h$ be an isometry of $(V, g)$ which preserves the leafwise tangent bundle $TF$, and denote by $H$ the compact (abelian) Lie group generated by $h$ in  $\Iso (V, g)$. Then the $K$-theory Lefschetz class of $h$ with respect to $(E, d)$ is the class
$$
L(h; E, d) := \Ind^H (E,d)_{h} := \frac{\Ind^H(E,d)}{1_{R(H)}}\; \in \; K^H(C^*(V, F))_h,
$$
obtained as the image of $\Ind^H(E, d)$ in the localized module $K^H(C^*(V, F))_h$ at the prime ideal in $R(H)$ associated with $h$. 
\end{definition}

\medskip

When the foliation is top-dimensional with  leaves given by the connected components of $V$, we recover the usual Lefschetz class as introduced and studied by Atiyah and Segal in \cite{AS2}. In general, the Lefschetz class can be related to topological data over the fixed point submanifold of $h$ with its potential induced foliation. More precisely, assume  that the fixed point submanifold $V^h$ is transverse to the foliation $F$ and denote by $F^h$ its induced foliation. Denote by $N^h$ the normal bundle to $V^h$ in $V$. It is then easy to check that the image   in the localized module $K_H(V^h)_h\simeq K(V^h)\otimes R(H)_h$ of the class 
$$
\lambda_{-1} (N^h\otimes \C) :=\sum_i (-1)^i [\Lambda^i (N^h\otimes \C)]\in K_H (V^h),
$$
is an invertible element for the obvious ring structure \cite{AS2}. 

\begin{theorem}\cite{B97}\label{K-Lef}\
Under the above notations and denoting by $i:TF^h\hookrightarrow TF$  the inclusion map, the following fixed-point formula holds in $K^H(C^*(V, F))_h$:
$$
L(h; E, d) = \left(\Ind _{(V^h, F^h)}\otimes \id_{R(H)_h}\right) \left( \frac{i^*[\sigma (E, d)]}{\lambda_{-1} (N^h\otimes \C)}\right), 
$$
where $\Ind_{(V^h, F^h)} : K (TF^h) \to K (C^*(V^h, F^h))$ denotes the Connes-Skandalis (topological) index morphism \cite{CS84} for the foliated manifold $(V^h, F^h)$. 
\end{theorem}

In the above formula, note that the RHS lives in $K(C^*(V^h, F^h))\otimes R(H)_h$ while the LHS lives in $K^H(C^*(V, F))_h$. We have thus implicitely used the quasi-trivial Morita extension associated with the transverse submanifold $V^h$, to view 
$$
K(C^*(V^h, F^h))\otimes R(H)_h \simeq K^H(C^*(V^h, F^h))_h \text{ as an $R(H)_h$-submodule of }K^H(C^*(V, F))_h.
$$
See \cite{CS84} as well as \cite{B97}.

\ms

Scalar Lefschetz formulae may be extracted from  Theorem \ref{K-Lef}  by pairing the above Lefschetz formula  with any $H$-equivariant cyclic cocycles over the smooth convolution algebra  $C_c^\infty (\maG)$ in the sense of \cite{BH10},  as explained above. An important class of examples is provided by cocycles which are associated with $H$-equivariant (even) Fredholm modules $(H, F)$ over the algebra $C_c^\infty (\maG)$ \cite{C85, BH10}. Such equivariant Fredholm modules are in turn  generated by geometric spectral triples  given by transversely elliptic Dirac-type operators which are $H$-invariant, see Section \ref{spec&Integrality}.

\ms

We now explain how to generate a large collection of $H$-equivariant cyclic cocycles when the space of leaves is $H$-trivial, meaning that it preserves the leaves and satisfies an extra natural condition. This latter condition is satisfied in all the examples we have in mind and by a large class of foliations. Moreover, when $H$ is connected, it is automatically satisfied for all foliations.  

\begin{definition}
A given diffeomorphism $f:V\to V$ which preserves the leaves of $F$, is a holonomy diffeomorphism (with respect to the foliation $F$) if  there exists a smooth map $\varphi^f: V \to \cG$, so that for any $x\in V$, $s(\varphi^f (x)) = x$, $r(\varphi^f (x)) = f(x)$, and the holonomy germ along $\varphi^f(x)$ coincides on small enough transversals with the action of $f$.
\end{definition}

\begin{remark}\
Theorem \ref{K-Lef} shows that the Lefschetz class $L(h; E,d)$ only depends on the restriction of all data to the $F$-saturation $\Sat (V^h)$ of $V^h$ in $V$, say to the smooth foliated open submanifold composed of those points whose leaves intersect $V^h$. But the action of the group $H$ restricts to this open foliated submanifold where it obviously preserves the leaves and  is automatically given by holonomy diffeomorphisms.
\end{remark}

In order to translate the Lefschetz formula of Theorem \ref{K-Lef} to cohomology, we restrict ourselves to the case where the group $H$ preserves each leaf of $(V, F)$ and acts by holonomy diffeomorphisms \cite{BH10}.  
We now recall the higher Lefschetz formula obtained in \cite{BH10} using any closed holonomy invariant Haefliger current $C$. The simplest example of such current in degree $0$ is a holonomy invariant transverse measure when this latter exists. But other interesting currents of any order provide interesting formulae as well. 
 Recall that $H$   acts on $(V, F)$ by holonomy diffeomorphisms which are isometries of $(V, g)$. 

\ms
Every closed even holonomy invariant current pairs with $H_c^{ev} (V/F)$,
the even Haefliger cohomology of $F$, \cite{Hae}.  
In \cite{BH10}, we defined the equivariant Connes-Chern character
$$
\ch^H: K^H (C_c^\infty (\cG, E)) \longrightarrow H_c^{ev} (V/F) \otimes C(H)^H,
$$
which extends to the equivariant $K$-theory of the completion $C^*$-algebra 
$K^H(C^*(V,F;E))\simeq K^H (C^*(V, F))$ and yields for general not necessarily abelian $H$:
$$
\ch^H: K^H (C^*(V, F)) \longrightarrow H_c^{ev} (V/F) \otimes C(H)^H.
$$
Moreover, we also proved in \cite{BH04} that any (even) holonomy invariant current $C$ produces a well defined even cyclic cocycle $\tau_C$ on the algebra $C_c^\infty (\maG; E)$ so that the pairing of the class $[\tau_C]$ with $K^H (C_c^\infty (\maG; E))$ given by Theorem \ref{Pairing} coincides with the composition of $\ch^H$ with evaluation against the homology class of $C$. Hence, when $H$ is the topologically cyclic abelian group generated by $h$, pairing   $\ch^H (\Ind_V^H(E,d))$  with the homology class of $C$, gives a continuous function on $H$ which can be evaluated  at $h$, to produce the  $C$-higher Lefschetz number $
L_C(h; E, d)$.
Said differently, 

\begin{definition}
Let $C$ be a closed even holonomy invariant current. The $C$-higher Lefschetz numbers $L_C(h; E, d))$ of $h\in H$ with respect to the $H$-equivariant leafwise elliptic complex $(E, d)$  is  
$$
L_C(h; E, d) := \langle \ch^H (\Ind_V^H(E,d)), [C]\rangle \; (h) \; \text{ a complex number in general.}
$$
\end{definition}

This higher $C$-Lefschetz number was expressed in \cite{BH10}  in terms of characteristic classes at the fixed points of  $h$.  Here is the precise statement.

\begin{theorem}\label{basic} (Higher Lefschetz Theorem,  \cite{BH10})
Assume that $H$ is a topologically cyclic compact Lie group of leaf-preserving holonomy diffeomorphisms which is topologically generated by the smooth diffeomorpism $h$. Moreover, assume that $F$ is oriented and that the fixed-point submanifold $V^h=V^H$ of $h$ is transverse to the foliation with the induced foliation $F^h.$  Denote by $i:TF^{h} \hookrightarrow TF$ the inclusion and by $N^h$ the normal bundle to $V^h$ in $V$. Then for any closed even dimensional holonomy invariant current $C$ on $(V,F)$, 
$$
L_{C}(h;E,d) \,\, = \,\, \Ind_{C|_{V^h}} \left(\frac{i^{*}[\sigma(E,d)](h)} {\lambda_{-1}(N^{h}\otimes{\C})(h)} \right).
$$
\end{theorem}

In this formula, $C|_{V^h}$ is the closed holonomy invariant current which is the restriction of $C$ to $(V^h, F^h)$;  and $\Ind_{C|_{V^h}}:K(TF^{h})\otimes \C \rightarrow \C$ is  the  higher $C|_{V^h}$-index map on the foliation $(V^h, F^h)$  defined as $
\Ind_{C|_{V^h}} = \tau_{C\vert_{V^h}} \circ (\Ind_{V^h}\otimes \id_\C)$ where $\Ind_{V^h}$ is the Connes-Skandalis index morphism for $(V^h, F^h)$. See  \cite{BH04} for more details.

\ms

Applying the Connes-Chern character, Theorem \ref{basic} gives the following  more computable expression.

\begin{theorem}\label{basic2}\   (Cohomological Lefschetz Formula, \cite{BH10})  Under the assumptions of Theorem \ref{basic}, the following formula holds
$$
L_C(h; E,d)=\left< \left[C|_{V^h}\right],  \int_{F^h} \frac{ch_{\C}(i^{*}[\sigma(E,d)](h))}{ch_{\C}(\lambda_{-1}(N^{h}
\otimes{\C})(h))} \, \Td(TF^{h} \otimes \C) \right >,
$$
where $\Td$ is the Todd characteristic class, $\ch_\C=\ch\otimes \id_\C$ with $\ch$ being the usual topological Chern character, and $\dd\int_{F^h}: H_c^* (V^h) \rightarrow H_c^{*-p} (V^h/F^h)$ is integration over the leaves of the fixed point foliation $(V^h, F^h)$.
\end{theorem}

An interesting situation is when $V^h$ is a strict transversal, , say with dimension equal to the codimension of the foliation $F$, which correspond for top-dimensional foliations to the case of isolated fixed points. In this case integration over $F^h$ disappears, $\Td(TF^{h} \otimes \C) = 1$ and $N^h = TF | _{V^h}$, so we get:

\begin{corollary}
Under the assumptions of \ref{basic2}, if $V^h$ is a strict transversal, then,
$$
L_C(h;E,d)  \,\, = \,\,  \left < [C|_{V^h}]  , \frac{ \sum_{i} (-1)^{i} \ch_{\C}([E^{i}|_{V^{h}}](h))}{\sum_{j} (-1)^{j} \ch_{\C}
([\wedge^{j}(TF | _{V^h} \otimes \C)](h))} \right >.
$$
\end{corollary}

We now describe our formula for Riemannian foliations when the holonomy invariant current is induced by a closed basic form. Recall that a form $\alpha$ on $V$ is basic with respect to $F$ if for any vector field $X$ tangent to $F$,  $i_X \alpha = 0$ and $i_X d\alpha = 0${, where $i_X$ is as usual  interior product with $X$}.  The {space} of such forms is denoted $ \cA^*_{bas}(V,F)$, and the associated {de Rham} cohomology is denoted $\oH^*_{bas}(V,F)$ {and called the basic cohomology of the foliation. It is well known that, for Riemannian foliations, the basic  cohomology is a honest  counterpart for de Rham cohomology of the space of leaves \cite{EKHS}.  It is well known for instance that 
$\oH^*_{bas}(V,F)$ satisfies Poincare duality if and only if $F$ is minimal (the mean curvature of the leaves of $F$ is zero), and that 
$\oH^q_{bas}(V,F)$ is either $0$ or $\R$, where $q$ is the codimension of $F$.  See \cite{ KT, Carr, EKHS}.
Given a basic form  $\alpha$, and a smooth compactly supported form $\beta \in \cA^*_c (T)$,  consider the complex number
$$
C_{\alpha}(\beta) \,\, :=  \,\,  \int_T  \beta \wedge \alpha |_T. 
$$ 
Then the complex number $C_{\alpha}(\beta)$ only depends on the Haefliger form $[\beta]$ represented by $\beta$ and maybe denoted $C_{\alpha}([\beta])$.  Indeed,  adding to $\beta$  forms of the type $h^*_{\gam} \what\beta - \what\beta$, does not alter the pairing, since basic forms are holonomy invariant, that is $h^*(\alp |_T) = \alp |_T$,  so $\dd \int_T \alpha |_T \wedge (h^*_{\gam} \what\beta - \what\beta) = 0$. Hence any basic form $\alpha$ gives rise to a holonomy invariant current $C_\alpha$}    In addition, it is easy to see that $dC_{\alpha} = C_{d\alpha}$, so $C_{\alpha}$ is closed if $\alpha$ is closed, and we end up with the induced map on homologies
$$
C:\oH^*_{bas}(V,F)  \longrightarrow   \oH_*(V/F),
$$
from basic cohomology to Haefliger homology.
\begin{remark}\
Note that $C$ is an isomorphism if the mean curvature of the leaves of $F$ is holonomy invariant \cite{KT}.
\end{remark}

 For the closed holonomy invariant current $C_\alpha$ given by a closed basic form $\alpha$,  applying Theorem \ref{basic2} gives, 
\begin{Equation}\hspace{2cm}
$\dd L_{C_{\alpha}}(h;E,d) \,\, = \,\, \int_{V^h}\frac{ch_{\C}(i^{*}[\sigma(E,d)](h))}{ch_{\C}(\lambda_{-1}(N^{h}\otimes{\C})(h))} \, \Td(TF^{h} \otimes \C) \wedge i^* \alpha,
$
\end{Equation}
\noindent
where $i:V^h \to V$ and $i:TF^h \to TF$.  
\color{black}

\ms

{Finally, we point out that in the general case all the computations of \cite{AS3} can be rewritten from our point of view by replacing the characteristic classes by their power series.   For example, recall from \cite{AS3}, p.\ 560, that the normal bundle $N^{h}$ decomposes under the orthogonal action of $h$ into
$$
N^{h}=N^{h}(-1) \oplus \sum_{0<\theta<\pi} N^{h}(\theta),
$$
where $N^{h}(-1)$ is a real bundle on which $h$ acts by multiplication by  $-1$, and each $N^{h}(\theta)$ is a sum of complex line bundles on which $h$ acts by multiplication by $e^{i\theta}$.
Let $s_1=dim_{\R}(N^{h}(-1))$, and  $x_{1}, \ldots, x_{[s_1/2]}$
be the standard characters which generate the Pontryagin dual of the maximal torus
of the orthogonal group ${\bf O}(s_1)$.  Let $s(\theta)=dim_{\C}(N^{h}(\theta))$, and $y_{1}, \ldots, y_{s(\theta)}$ be the corresponding characters for unitary group ${\bf U}(s(\theta))$.
Set
$$ 
\begin{array}{ccccl}
{\mathcal{R}}&=& \sum {\mathcal{R}}_{r}(p_{1}, \ldots ,p_{r}) &=&
 \left [ \prod_{j=1}^{[s_1/2]} ((1+e^{x_{j}})/2) \, ((1+e^{-x_{j}})/2) \right ] ^{-1}\\
{\mathcal{S}}&=& \sum {\mathcal{S}}_{r}^{\theta}(c_{1}, \ldots ,c_{r}) &=&
\left [ \prod_{j=1}^{s(\theta)} (1-e^{y_{j}+i\theta})( 1-e^{-y_{j}-i\theta})/
((1-e^{i\theta})(1-e^{-i\theta}))\right ] ^{-1}\\
\end{array}
$$
where $p_{i}$ is the i$th$ symetric function of the $x_{i}'s$ (i.\ e.\ a Pontryagin class)
and $c_{i}$ is the i$th$ symetric function of the $y_{i}'s$ (i.\ e.\ a Chern class).
Then we have 
\begin{theorem}\label{basic3} Under  the assumptions of Theorem 
\ref{basic2},
$$
L_C(h;E,d)=\left < \frac {ch_{\C}(i^{*}[\sigma(E,d)](h))}{det(1-h| \, {N^{h}})} \,
\prod_{0< \theta < \pi} {\mathcal{S}}^{\theta}(N^{h}(\theta)) \, {\mathcal{R}}(N^{h}(-1)) \, Td(TF^{h} \otimes \C), [C|_{V^h}] \right >.
$$
\end{theorem}

For the four classical complexes (deRham, Signature, Spin, and Dolbeault), there are refinements of this general formula.  We refer the reader to \cite{AS3} for those formulas.  

\section{Lefschetz Examples}\label{LefEGs}

\subsection{The de Rham complex}
For the deRham complex, things are particularly simple, and we begin this section with some general results for this case.   In particular, if (E,d) is the de Rham complex along the leaves, then only the zero component of the  closed holonomy invariant current $C$ is involved.  To see this, we compute $\dd\frac{i^{*}[\sigma(E,d)](h)}{\lambda_{-1}(N^{h}\otimes{\C})(h)}$  in this case.
Now $\sigma(E,d)$ is given by the sequence
$$
0 \longrightarrow \Lambda^0_{\C} T^*F \stackrel{\wedge \xi}{\longrightarrow} \Lambda^1_{\C} T^*F \stackrel{\wedge \xi}{\longrightarrow}\Lambda^2_{\C} T^*F \stackrel{\wedge \xi}{\longrightarrow} \cdots
$$
where $\xi \in T^*F$ and $ \Lambda^k_{\C} T^*F =  \Lambda^k T^*F \otimes \C$ is the complexified $k$-th exterior power.  When we restrict this sequence to $F^h$, we get 
$\dd \Lambda^k T^*F \,\, = \,\, \dd\oplus_{i+j=k} \Lambda^i T^*F^h \otimes\Lambda^j N^{h,*},$
and $\xi \in T^*F^h$ acts only on the first factor.  Thus we have that 
$$
i^{*}[\sigma(E,d)] \,\, = \,\,   \sigma(E^h,d) \lambda_{-1}(N^{h}\otimes{\C}),
$$
where $(E^h,d)$ is the de Rham complex along $F^h$.  So,
$$
\frac{i^{*}[\sigma(E,d)](h)}{\lambda_{-1}(N^{h}\otimes{\C})(h)} \,\, = \,\, \sigma(E^h,d)(h).
$$
Now $\ch_{\C}(\sigma(E^h,d)(h)) \Td(TF^{h} \otimes \C)$ is just $\chi(TF^h)$, the Euler class of $TF^h$, 
which is non-zero only in dimension equal to the dimension of $F^h$.  
Thus  $\dd \int_{F^h}\chi(TF^h)$ is a zero dimensional Haefliger cohomology class, where $ \dd \int_{F^h}$ is Haefliger's integration over the plaques of $F^h$, \cite{Hae}.  But, 
$$
L_C(h;\text{de Rham})\,\, :=  \langle  \left(\int_{F^h}\chi(TF^h) \right ), [C|_{V^h}]\rangle \,\, = \,\, \langle \left ( \int_{F^h}\chi(TF^h) \right ), [C_{0}]\rangle,
$$
where $C_{0}$ is the zero component of $C|_{V^h}$, a holonomy invariant transverse distribution.  So the higher components of $C$ do not contribute to the Lefschetz formula for the de Rham complex. As an obvious corollary, $L_C(h;\text{de Rham})$ vanishes for all holonomy invariant transverse currents without  zero component. When the foliation is oriented for instance, and denoting by $[V/F]$  the class of the  transverse fundamental cycle, we get
 $$
L_{[V/F]}(h;\text{de Rham})=0.
$$
 
\ms

{In the special case where $V^h$ is a strict transversal, the leafwise tangent bundle $TF^h$ reduces to the zero bundle and we get:
$$
 \frac{ \sum_{i} (-1)^{i} \ch_{\C}([E^{i}|_{V^{h}}](h))}{\sum_{j} (-1)^{j} \ch_{\C}
([\wedge^{j}(TF^{h} \otimes \C)](h))} \,\, =\,\,
\frac{\sum_{i} (-1)^{i} \ch_{\C}([\wedge^{i}(TF^{h} \otimes \C)](h))}{\sum_{j} (-1)^{j} \ch_{\C}
([\wedge^{j}(TF{h} \otimes \C)](h))} \,\, =\,\, 1.
$$
Hence,  in the case where  $C_0$ is a positive holonomy invariant 
transverse 
measure, one gets  
$$
L_C(h;\text{de Rham}) = L_{C_0}(h;\text{de Rham}) = C_0(V^h)  \geq 0,
$$  compare with \cite{HL90}.   
Second,
if  the foliation is transversally oriented,  the zero-th component of the transverse fundamental class is zero, \cite{C86}.  More specifically, 
$$
L_{[V/F]}(h;\text{de Rham})=0.
$$
To see this, recall that for Haefliger forms,
$\dd [[V/F]|_{V^h}] \left (\omega \right ) = 0$, unless
the degree of $\omega = q$.  As $q > 0$,  $$
L_{[V/F]}(h;\text{de Rham}) \,\, = \,\, \left < \chi(TF^h), [[V/F]|_{V^h}]\right > \,\, = \,\,  [[V/F]|_{V^h}]\left ( \int_{F^h}\chi(TF^h)\right ) \,\, = \,\,  0.
$$} 

\subsection{A Universal Example}
The following is an extension of Section 4 in \cite{HL90}, which provides non trivial Lefschetz formulae for the four classical geometric operators.   We will refer to that material freely.   In this case, the higher terms of the Haefliger Lefschetz class are non-trivial for the four classical geometric operators, if we  twist them by an appropriate universal leafwise flat bundle.   

\ms
 
Consider the following ({\em universal}) $\C$ bundle over the torus $\T^{2k} = \R^{2k}/ \Z^{2k}$. 
Let $(x,y) = (x_1,y_1, \cdots, x_k,y_k)$ be coordinates on $\R^{2k}$, and $(m,n) = (m_1, n_1\ldots,m_k, n_k) \in \Z^{2k}$. 
Set 
$$
W =  (\R^{2k} \times \C) /  \Z^{2k}, 
$$
where $(m,n)$ acts on  $(x,y,z)$ by
$$
(m,n) \cdot (x,y,z) \,\, = \,\,  (x_1 + m_1,y_1 + n_1, \cdots, x_k + m_k,y_k +n_k, \exp(2\pi i\sum n_j x_j) z). 
$$
If we denote by $\pi_j:\T^{2k} \to \T^2$ the projection $\pi_j(x,y) = (x_j,y_j)$, then $W = \otimes_j \pi_j^* \what{W}$, where $\what{W}$ is the  bundle over $\T^2$ given by   
$$
\what{W} \,\, = \,\,  [0,1] \times [0,1] \times \C / \sim,
$$
where $(x,0,z) \sim (x,1,\exp(2\pi i x)z)$ and $(0,y,z) \sim (1,y,z)$.  It is a straight forward calculation 
that $\ch(\what{W}) = 1 + \eta \beta$, where $\eta$, $\beta$ is the natural basis of $\oH^1 (\T^{2};\R)$.  It follows immediately that if 
$\eta_1,\beta_1, ...,\eta_k, \beta_k$ is the natural basis of $\oH^1 (\T^{2k};\R)$,  then

\begin{proposition} \label{Luszch}
\hspace{2cm}  $\dd \ch (W)  \,\,=\,\, \prod_{i=1}^k (1 + \eta_i \beta_i).$
\end{proposition}

The surface $\Sigma_4$ of genus $4$ has fundamental group $\Gamma$ generated by eight isometries $(\alpha_j)_{0\leq j \leq 7}$ of the Poincar\'e disk $\H$ with the relation 
$$
\alpha_0 \alpha_1^{-1} \alpha_2 \alpha_3^{-1}  \alpha_4 \alpha_5^{-1}  \alpha_6 \alpha_7^{-1}
\alpha_0^{-1} \alpha_1 \alpha_2^{-1}  \alpha_3 \alpha_4^{-1}  \alpha_5 \alpha_6^{-1}
\alpha_7 \,\, = \,\,Id.
$$
The element  $\alpha_j  = \theta^{-j} \alpha \theta^{j} \in SL_2\R$, where 
$\alpha=\left(\begin{array}{cc} d & 0 \\ 0 & d^{-1} \end{array}\right)$,
for proper choice of $d > 0$, and $\theta$ is rotation by $\pi/16$.  Then
$\Sigma_4 = \Gamma\backslash SL_2\R / SO_2$, and we  may take for a fundamental domain of $\Sigma_4$ a regular $16-$gon $\D$ centered at $0 \in \H$.  The action we have chosen for $\Gamma$ identifies opposite edges of $\D$ by translation along the geodesic through the midpoints of the respective edges. 

\ms

Choose $a, b \in \R$, and denote by $A$ and $B$ the translations of $\R^{2k}$ given by  $A(x,y) = (x, y+a)$ and  $B(x,y) = (x, y+b)$.  For $j = 0,3,4,7$ set $f(\alpha_j) = A$, and for $j = 1,2,5,6$ set $f(\alpha_j) = B$.  This induces a homomorphism $f:\pi_1(\Sigma_4) \to \Diff(\T^{2k})$, and so a foliated bundle $V = (SL_2\R /SO_2) \times_f \T^{2k}$, with foliation $F$.    If $a$ and $b$ are not rationally related, i.e. for all $(k, \ell) \in \Z \times \Z - (0,0)$, $ka + \ell b \in \R - \Q$, then all the leaves of the foliation are isomorphic to $\H^2$.  Note that the action given by $f$ preserves $W$, and that $W_V$, the pull back of $W$ to $V$, is flat along the leaves of $F$.
It is also not difficult to see  $V \simeq \Sigma_4 \times \T^{2k}$, so
$$
\oH^*(V;\R) \,\,=\,\, \oH^*(\Sigma_4;\R) \otimes \oH^*(\T^{2k};\R).
$$

\ms

We denote a point in $V$ by $[gSO_2,t]$ where $g \in SL_2\R$ and $t \in \T^{2k}$.  Let $r \in SO_2$ be rotation by $\pi/4$, and note that the action of $r$ on $\D$ is rotation by $\pi/2$.  Define
$h:V \to V$ by 
$$
h([gSO_2,t]) \,\, = \,\, [rgSO_2,t].
$$
Then, as $f(\alpha_j) = f(\alpha_{j+4})$ mod $8$, it follows just as in \cite{HL90} that 
$h$ is well defined and that it also preserves $F$.  In addition,  the fixed point set of $h$ consists of the two fibers of $V$ over the points $v_0$ and $v_1$ in $\Sigma_4$ corresponding to
$0$, the center of $\D$, and the $16$ vertices on the boundary of $\D$, which are all identified when we glue $\D$ to get $\Sigma_4$.  Thus $V^h = \T^{2k}_0  \cup \T^{2k}_1$, and just as in \cite{HL90}, the action of $h$ at these points is rotation by $\pi/2$ in the $\Sigma_4$ direction, and the identity in the $\T^{2k}$ direction.  

\ms

As $V^h$ is a union of strict transversals, we will be applying the second part of Theorem \ref{basic2}, so we need to calculate $\sum_{j} (-1)^{j} \ch_{\C} ([\wedge^{j}(F \, |_{V^{h}} \otimes \C)](h))$,
and $\sum_{i} (-1)^{i} \ch_{\C}([E^{i} \, |_{V^{h}}](h))$ for each of the various leafwise complexes we consider.   We will do the computations only on $\T^{2k}_0$ as they are the same at $\T^{2k}_1$.
Note that  $TF \, |_{\T^{2k}_0}$ is a trivial $\R^2$ bundle over $\T^{2k}_0$ and the action of $h$  on the fibers of $TF \, |_{\T^{2k}_0}$ is rotation by $\pi/2$.  Thus
$$
\sum_{j} (-1)^{j} \ch_{\C} ([\wedge^{j}(TF \, |_{\T^{2k}_0} \otimes \C)](h)) \,\, = \,\, \det((1-h) \, |_{TF \, |_{\T^{2k}_0} \otimes \C}) \,\, = \,\, 2. 
$$
Next, note that the twisting bundle  $W_V \, | \,  \T^{2k}_j = W$ for $j = 0,1$.  The action of $h$ on $W$ is the identity, so at each $\T^{2k}_j$,
$$
\ch([W_V |_{\T^{2k}_j}](h) = \ch (W)  \,\,=\,\, \prod_{i=1}^k (1 + \eta_i \beta_i).
$$
The expression in the numerator of the formula in the second part of Theorem \ref{basic2} is of the form 
$$
\sum_{i} (-1)^{i} \ch_{\C}([E^{i} \otimes W_V \, |_{V^h}](h) \,\,  =  \,\,  
\ch([W_V |_{V^h}](h)\sum_{i} (-1)^{i} \ch_{\C}([E^{i} \, |_{V^h}](h) \,\, = \,\,
$$
$$
\prod_{i=1}^k (1 + \eta_i \beta_i)\sum_{i} (-1)^{i} \ch_{\C}([E^{i} \, |_{V^h}](h),
$$ 
so we need only compute $\sum_{i} (-1)^{i} \ch_{\C}([E^{i} \, |_{V^h}](h)$ 
for the four geometric complexes.  This is quite simple since each $E^{i} \, |_{V^h}$ is a trivial bundle, in fact the pull-back of the fiber of the associated bundle on $\Sigma_4$ at the base point of the fixed fiber.  It is then easy to see that  $\sum_{i} (-1)^{i} \ch_{\C}([E^{i} \, |_{V^h}](h)$ is just the pull-back of the fixed point index at that point for the action of $h$ on the related complex on $\Sigma_4$.  These were computed in  \cite{HL90}.

\subsubsection{deRham Complex}
$$
\sum_{i} (-1)^{i} \ch_{\C}([E^{i} \, |_{V^h}](h) \,\, = \,\,  2, 
$$
so
$$
L_C(h;\text{deRham}(W_V)) \,\, = \,\,  
\sum_{j=0}^1 \left <  \prod_{i=1}^k (1 + \eta_i \beta_i),C|_{\T^{2k}_j} \right >  \,\, = \,\,  2  \left <  \prod_{i=1}^k (1 + \eta_i \beta_i),C|_{\T^{2k}_0} \right >,
$$
since the Haefliger current $C$ is invariant under holonomy.
It is easy to see that each term in the expression $ \prod_{i=1}^k (1 + \eta_i \beta_i)$ has a Haefliger current which is dual to it.

\subsubsection{{Signature Complex}}
$$
\sum_{i} (-1)^{i} \ch_{\C}([E^{i} \, |_{V^h}](h) \,\, = \,\, -2 i,
$$
so
$$
L_C(h;\text{Signature}(W_V)) \,\, = \,\,   \frac{- 2 i }{2}  \sum_{j=0}^1 \left <  \prod_{i=1}^k (1 + \eta_i \beta_i),C|_{\T^{2k}_j} \right >  \,\, = \,\,   - 2 i \left <  \prod_{i=1}^k (1 + \eta_i \beta_i),C|_{\T^{2k}_0} \right >.
$$
 
\subsubsection{{Dolbeault Complex}}
Since $h$ is a holomorphic map on each leaf, it induces endomorphisms of the Dolbeault leafwise complexes (for $k = 0,1$)
$$
0 \longrightarrow C^{\infty}(\wedge^k T^*F \otimes_{\C} \wedge^0 \overline{T}^*F)\longrightarrow   
C^{\infty}(\wedge^k T^*F \otimes_{\C} \wedge^1 \overline{T}^*F)
\longrightarrow 0,
$$
where $ T^*F$ and $\overline{T}^*F$ are the holomorphic and anti-holomorphic cotangent bundles respectively.  We denote these endomorphisms by $\what{h}_0$ and $\what{h}_1$, respectively.   Then 
$$
\sum_{i} (-1)^{i} \ch_{\C}([E^{i} \, |_{V^h}](\what{h}_j) \,\, = \,\, i + 1 - 2j,
$$
so
$$
L_C(\what{h}_j;\text{Dolbeault}(W_V)) \,\, = \,\,  
 \frac{ i + 1 - 2j}{2}  \sum_{j=0}^1 \left <  \prod_{i=1}^k (1 + \eta_i \beta_i),C|_{\T^{2k}_j} \right >  \,\, = \,\,   ( i + 1 - 2j) \left <  \prod_{i=1}^k (1 + \eta_i \beta_i),C|_{\T^{2k}_0} \right >.
$$

\subsubsection{{Spin Complex}}
Using the two liftings, denoted $\wtit{h}_\pm$, defined in  \cite{HL90}, we have 
$$
\sum_{i} (-1)^{i} \ch_{\C}([E^{i} \, |_{V^h}](\wtit{h}_\pm) \,\, = \,\, \pm{i}{\sqrt{2}},
$$
so
$$
L_C(\wtit{h}_\pm;\text{Spin}(W_V)) \,\, = \,\,  
 \frac{\pm{i}{\sqrt{2}}}{2}  \sum_{j=0}^1 \left <  \prod_{i=1}^k (1 + \eta_i \beta_i),C|_{\T^{2k}_j} \right >  \,\, = \,\,   \pm{i}{\sqrt{2}}\left <  \prod_{i=1}^k (1 + \eta_i \beta_i),C|_{\T^{2k}_0} \right >.
$$

As $k$ is as large as we please and each term in the expression $ \prod_{i=1}^k (1 + \eta_i \beta_i)$ has a Haefliger current which is dual to it, we thus have universal non-triviality for the higher Lefschetz numbers for all four of the classical elliptic complexes. 

\subsection{Some cancellation relations}
{We briefly explain in this paragraph the close relation between Lefschetz formulae and some standard series relations.   We concentrate here on a single example which shows how the triviality of the higher terms in our Lefschetz formula can be used already for simple foliations, to rediscover surprising  cancellation identities. This approach opens up the way for other more involved relations by using more complicated foliations.}.  This application is thus in the spirit of the by now classical results which relate the Lefschetz theorem with number theory, see for instance \cite{HirzebruchZagier}.  {Recall that for any formal variable $z$,
$$
\coth (z) =
\frac{e^z + e^{-z}}{e^z - e^{-z}} =  \frac{1}{z} + \sum_{\ell \geq 1} \, \frac{2^{2\ell}}{(2\ell)!} \,b_{2\ell} \; z^{2\ell-1},
$$
where the $b_n$ are the Bernoulli rational numbers.}

\begin{corollary}
For any integer $n\geq 1$ and any sequence of angles $0\leq  \alpha_0 \leq \cdots \leq \alpha_n <\pi/2$, the formal series in the commuting (formal free) variables $z_0, \cdots, z_n$:
$$
\sum_{j=0}^n  \prod _{j \neq k=0}^n \coth \left[z_k-z_j+i(\alpha_k -\alpha_j)\right] \,\, = \,\,  \frac{1 + (-1)^{n} }{2},
$$
where $\coth \left[z_k-z_j+i(\alpha_k -\alpha_j)\right]$ is the power series 
$\dd\frac{\coth (z_k-z_j) + i \tan (\alpha_k-\alpha_j)}{1 + i \tan (\alpha_k-\alpha_j) \coth (z_k-z_j)}.$
Said, differently, we have for any $n\geq 1$ and any commuting (formal) variables $z_0, \cdots, z_n$
$$
\sum_{j=0}^n  \prod _{j \neq k=0}^n \frac{\coth (z_k-z_j) + i \tan (\alpha_k-\alpha_j)}{1 + i \tan (\alpha_k-\alpha_j) \coth (z_k-z_j)} \,\, = \,\,  \frac{1 + (-1)^{n} }{2}.
$$
\end{corollary}

Note that the corollary can be translated into an infinite number of relations corresponding to the coefficients of the series. It is also easy to see that it is equivalent to the simpler formula for the power series in the 2 by 2 distinct (for instance complex) variables  $u_0, \cdots, u_n$:
$$
\sum_{j=0}^n  \prod _{j \neq k=0}^n \coth (u_k-u_j) \,\, = \,\,  \frac{1 + (-1)^{n} }{2}.
$$
However the statement makes the proof more understandable. Also notice that these relations can  be derived from a straightforward (rather astute) computation, but we prove it here as an {easy} application of the family Lefschetz theorem.
 
\begin{proof}  
We only need to  give the proof of the formula under the assumption $0<\alpha_0 < \cdots < \alpha_n <\pi/2$. Set $B = \C P_{k_0} \times \cdots \times \C P_{k_n}$, $k_j \geq n$, which has  cohomology  the  polynomial algebra $\R[x_0, \cdots, x_n]$ on two dimensional generators, truncated by the  relations
$x_j^{k_j+1} = 0.$
Denote by $E_0, \cdots, E_n$ the natural $\C$ line bundles over $B$,  so $c_1 (E_j) = x_j$.  Denote by $P$ the principal $U_1 \times \cdots \times U_1 \simeq \T^{n+1}$ bundle associated to $\oplus_j E_j$. The torus $\T^{n+1}$ acts on $\C P_n$  by
$$
(u_0,\ldots ,u_n)[z_0:\cdots :z_{n}] \,\, := \,\, [u_0 z_0 : \cdots : u_n z_n].
$$
Set $V = P \times_{\T^{n+1}} \C P_n$, the quotient of $P \times \C P_n$ by the diagonal action of $\T^{n+1}$.
Choose $a=(a_0, \ldots , a_n) \in T^{n+1}$ with $a_i = \exp(i\theta_j)$, where  $0 < \theta_0 < \cdots < \theta_n <\pi$.  Define the fiberwise action of $a$ on $V$ by
$$
a [p, [z_0:\cdots : z_n]] := [p, [ a_0z_0 : \cdots : a_n z_n]].
$$
Note that the fibers of $\pi:V\to B$ are oriented by their complex structures and that $a$ preserves this orientation. The fixed point submanifold of $a$ is the union of $n+1$ connected components $(B_j)_{0\leq j \leq n}$ where $B_j$ is given by
$$
B_j = \{ [p, [0:\cdots : 0 : 1 : 0 : \cdots : 0]] \, | \, p\in P\} \text{ with } 1 \text{ at the }j-th\text{ position.}
$$
Each $B_j$ is a transversal to the fibration, and is diffeomorphic to $B$ under $\pi$. The normal bundle to $B_j$ is the pull-back under $\pi:B_j \to B$ of the vector bundle $ \oplus_{j\neq k=0}^n ( E_k\otimes E_j^*)$. Moreover, this decomposition into complex line bundles corresponds to the decomposition of the normal bundle of $B_j$ into the eigenspaces of the action of the isometry $a$. We now apply the Lefschetz fixed point formula for families \cite{Bena02} to $a$ with respect to the signature operator $D^+_{vert}$ along the oriented even dimensional fibers of $V$.   We get 
$$
\Sign (a) \,\, = \,\, \ch(L(a;D^+_{vert}))    \,\, = \,\,  \sum_{j=0}^n  \pi_*  \nu_j \quad \in \quad \oH^*(B;\R),
$$
where  $L(a;D^+_{vert}) \in \oK(B) \otimes \C$ is the  Lefschetz class of the index of $D^+_{vert}$ evaluated at $a$.  The expression $\nu_j \in \oH^*(B_j;\R)$ is the local contribution corrsponding to $B_j$ in the fixed point formula. More precisely, denote by $\maM^\theta$ the multiplicative sequence associated with the series in the $x$ variable
$$
\frac{   \tanh(i\theta/2)}{\tanh ((x+i\theta)/2)}.
$$
Then
$$
\pi_* \nu_j   \,\, = \,\,   
\prod_{j \neq k = 0}^n  \left(i \tan (\frac{\theta_k-\theta_j}{2})\right)^{-1} \maM^{\theta_k-\theta_j} ( E_k\otimes E_j^*) \,\, = \,\, 
\prod _{j \neq k=0}^n \coth \left(\frac{x_k-x_j+i(\theta_k -\theta_j)}{2}\right).  
$$

The Lefschetz class $\Sign(a)$ can also be computed as follows.  Denote the fiberwise cohomology of $V \to B$ by $\oH^*(V \, | \,B)$, and the $\pm$ eigenspaces of fiberwise involution associated to the signature operator by $\oH^*_{\pm}(V \, | \,B)$.
Since $\T^n$ is connected, all elements act by the identity on this cohomology.  Thus $\Sign(a) = \Sign(I) = \ch(\oH^n_+(V \, | \,B)) - \ch (\oH^n_-(V \, | \,B))$.  When n is odd, $\oH^n(V \, | \,B)) = 0$, so $\Sign(a) = 0$.  When n is even, $\oH^n(V \, | \,B)$ is one dimensional, and   
$$
\oH^n_+(V \, | \,B) \,\, = \,\, \oH^n(V \, | \,B)  \quad \text{and} \quad \oH^n_-(V \, | \,B) \,\, = \,\, 0.
$$
Thus 
$$
\Sign(a) \,\, = \,\, \ch(\oH^n(V \, | \,B))  \,\, = \,\, 1 \quad \in \quad  \oH^0(B),
$$
since  the line bundle $\oH^n(V \, | \,B)$ over $B$ is a trivial one dimensional  bundle.  
Therefore, in $\oH^*(B;\R)$,
$$
\sum_{j=0}^n\prod _{j \neq k=0}^n \coth \left(\frac{x_k-x_j+i(\theta_k -\theta_j)}{2}\right) \,\, = \,\,  \frac{1 + (-1)^{n} }{2}.
$$
Replacing $x_k$ by $2z_k$ and $\theta_k$ by $2\alpha_k$ and noting that
the dimensions of the $\C P_{k_j}$ are as large as we want completes the proof.
\end{proof} 

\section{{Obstructions to group actions}}

For the {spin Dirac operator along the leaves of a foliation, we can generalize the Atiyah-Hirzebruch Rigidity Theorem  \cite{AH} and prove a higher rigidity theorem which extends the measured rigidity theorem proven in \cite{HL91}.} 

\subsection{The higher rigidity theorem}

Suppose that  $TF$ is an even dimensional spin bundle, with associated leafwise Dirac operator $D$.  Let $\what{A}(F)$ be the $\what{A}$-characteristic class of the vector bundle $TF$.
Using our higher Lefschetz theorem, we generalize Proposition 3.2 of \cite{HL91} by taking into account all the closed holonomy invariant currents  \cite{Hae}. 

\begin{theorem}\label{rigidity}\ 
Assume that the tangent bundle $TF$ to the foliation  $(V,F)$ is spin and that there exists a closed holonomy invariant current $C$ such that 
$$
\left< \int_F \what{A}(TF), {[C]}\right> \neq 0.
$$
Then no compact connected Lie group $H$ can act non trivially as a group of isometries of $V$ preserving the leaves of F and their spin structure.
\end{theorem}

Recall that since $H$  is connected and preserves each leaf, the fixed point submanifold is  transverse to F, \cite{HL90}. Moreover, in this case all the elements of $H$ act by  holonomy diffeomorphisms. If $C$ is given by a holonomy invariant transverse measure, we get the rigidity theorem of \cite{HL91}.  The power of Theorem \ref{rigidity} is that  applies to a potentially large alternative collection of invariant currents \cite{Hae}. The proof uses the continuity of the higher Lefschetz map $h\rightarrow L_{C}(h;E,d)$ and a straightforward generalization of a classical method based on Liouville's theorem \cite{AH} together with our higher Lefschetz fixed point formula.  A different proof given in the beautiful paper of Lawson and Yau \cite{LawsonYau}, with better results for effective $\S^3$ actions, should also be extendable to foliations.

\ms

The proof of Theorem \ref{rigidity} relies on  the following proposition. Notice that this proposition applies in particular to all leafwise actions of compact connected Lie groups.

\begin{proposition}\label{continuity}  Let $F$ be an oriented foliation of the compact  manifold $V$. 
 Assume that the compact  Lie group $H$ acts by holonomy diffeomorphisms on $(V,F)$. Then for any  $H$-equivariant leafwise elliptic pseudodifferential complex $(E,d)$ over $(V,F)$, and every closed  Haefliger $2k$-current $C$ on $(V,F)$, the map  
$h \mapsto L_{C}(h;E,d)$ is a continuous map from $H$ to $\C$.  
\end{proposition}

Assuming Proposition \ref{continuity}, we give the proof of Theorem \ref{rigidity}.

\begin{proof}  We follow the proof of \cite{HL91}, making the necessary changes and correcting some errors. We need to extend the classical proof  of Atiyah-Hirzebruch \cite{AH} to our foliation setting.  We may asume that $H = \S ^1 \subset \C$.  By Lemma 3.3 of \cite{HL91}, the fixed point set $V^H$ of $H$ is a closed submanifold of $V$, which is transverse to $F$.  Note that $V^H$ is the fixed point set of any topological generator of $H$, and  $V^H \neq V$, since $H$ acts non-trivially.  If $V^H = \emptyset$, we have immediately that $
\dd \left< \int_F \what{A}(TF), C\right> = 0,
$
for all $C$, a contradiction, {so in summary, $V^H$ is a  transverse submanifold such that} $V \neq V^H \neq \emptyset$.

\ms

Denote by $V^H_{\alpha}$ a  connected component of $V^H$ with the induced foliation $F_{\alpha}$.  If $y \in V^H_{\alpha} \cap L$, then the normal bundle $N$ to $V^H_{\alpha} \cap L$ in $L$ {(=normal bundle to $V^H$ in $V$)} at $y$ can be written as $\oplus N^j_y$, where $H$ acts on $N^j_y$ by the representation $z \mapsto z^{m_j}$ for some positive integer $m_j$.  It follows that the $N^j$ are complex $H$ bundles on $V^H_{\alpha} \cap L$. 

\ms

Let $z \in \C$ and  $x \in \R$.  The function $1/(1-ze^{-x})$ can be written as $R(x,z)$, which is a formal power series in $x$ whose coefficients are rational functions of $z$, having a pole only at $z = 1$, and no pole at $z = \infty$.  In particular, the coefficient of $x^0$ is $1/(1-z)$, and that of $x^n$ for $n >0$  is $((-1)^n/n!)\sum_{k=1}^{\infty}k^n z^k$.  It is an easy exercise to prove that this is
$$
\frac{(-1)^n}{n!}  [z  [\cdots [z[\frac{1}{1-z}]']' \cdots ]'],
$$
where the $'$ indicates differentiation, {and it} is done $n$ times.

\ms

Let $A^r(z) = (z^r)^{1/2}\prod_{\ell=1}^r e^{-x_{\ell}/2}R(x_{\ell},z)$.  Because of the factor $(z^r)^{1/2}$, this is defined only up to sign.  It defines a formal power series in $\sigma_1,..,\sigma_r$ the elementary {symmetric} functions in $x_1,...,x_r$. 

\ms

Now any $z \in H$ which is a topological generator acts on $N^j$ by multiplication by $z^{m_j}$. For such $z$, let $A(N^j,z) = A^{d_j}(z^{m_j})$, where $d_j$ is the (complex) dimension of $N^j$.
The Riemannian connection on $V^H_{\alpha} \cap L$ preserves the bundles $N^j$ and is a complex connection on each $N^j$.  We can replace the {symmetric functions} $\sigma_1,..,\sigma_r$ by the corresponding Chern polynomials in the curvature of this connection.  In this way $A(N^j,z)$ becomes  a differential form on $V^H_{\alpha} \cap L$.  Let $\what{A}(V^H_{\alpha} \cap L)$ be the differential form we get by replacing the Pontrjagin classes in the expression for the $\what{A}$ class of a real vector bundle by the Pontrjagin polynomials in the curvature of the Riemannian connection on $V^H_{\alpha} \cap L$, {see for instance \cite{AS3}, p.\ 570}.  Then define
$$
A(V^H_{\alpha} \cap L,z) \,\, = \, \, (-1)^{(\dim_{\R}F)/2} \what{A}(V^H_{\alpha} \cap L) \prod_j A(N^j,z).
$$
Now $A(V^H_{\alpha} \cap L,z)$ is a differential form on $V^H_{\alpha} \cap L$ which has the factor $(z^d)^{1/2}$, where $d = \sum m_jd_j$ and $d_j = \dim_{\C} N^j$.  The choice of sign is determined by {adapting} the argument of \cite{AH}, p.\ 21.  Namely, let $y \in V^H_{\alpha}$ and choose $(z^d)^{1/2}$ so that 
$$
2^{n_{\alpha}}(z^d)^{1/2} \prod_j (1 + z^{-m_j})^{d_j} \,\, = \,\, \text{trace}(z \, | \, E^+_y \oplus E^-_y)),
$$
where $L$ is the leaf through $y$, $2{n_{\alpha}} = \dim(V^H_{\alpha} \cap L)$, and $E^{\pm}$ are the complex vector bundles associated to the spin structure constructed from the irreducible spin representations $\Delta^{\pm}$.  Then for $z$ a topological generator of $H$, $\what{A}(V^H_{\alpha} \cap L,z)$ is the differential form given on the right hand side of the equation in Theorem \ref{basic2}  for the  {spin} Dirac complex $(E,d)$, so we have
$$
L_C(z; E,d)   \,\,=\,\,  \left<  \frac{ch_{\C}(i^{*}[\sigma(E,d)](z))}{ch_{\C}(\lambda_{-1}(N^{z}
\otimes{\C})(z))} \, \Td(F^{z} \otimes \C) , C \, |_{V^z} \right >   \,\,=\,\, 
\sum_{\alpha} \left< \int_{F_{\alpha} }  \what{A}(V^H_{\alpha} \cap L,z), C \, | \, _{V^H_{\alpha}}   \right> .
$$

Now consider the function {\em on the complex plane} given by
$$
A(F,z)   \,\,=\,\, 
\sum_{\alpha} \left<  \int_{F_{\alpha} }  \what{A}(V^H_{\alpha} \cap L,z),  C \, | \, _{V^H_{\alpha}}  \right> .
$$

Since  $\what{A}(V^H_{\alpha} \cap L,z)$ is a differential form whose coefficients are rational functions having poles only at roots of unity, the same is true of  $\dd \int_{F_{\alpha} }  \what{A}(V^H_{\alpha} \cap L,z)$, so also of $A(F,z) $.  Because of the factor $(z^d)^{1/2}$, $A(F,0) = 0$, and $A(F,z)$ has no pole at $z = \infty$.   
If $z \in H$ is not a root of unity, then it is a topological generator, and we have $L_C(z; E,d) = A(F,z)$.  Now  $L_C(z; E,d)$ is defined for all $z \in H$, and  by Proposition 
\ref{continuity} it is a continuous function on $H = \S^1$.  Thus $A(F,z)$ has no poles and so is analytic and bounded, and thus constant and hence zero. Therefore $L_C(z; E,d) =0$ for all $z \in \S^1$.  But $\dd L_C(1; E,d) =\left< \int_F \what{A}(TF), C\right>$.
\end{proof}

\begin{remark}
Note that any closed basic differential form yields a closed holonomy invariant current. So when the foliation is for instance Riemannian,  the subring of $H^*(V, \C)$ generated by the Pontryagin classes of the normal bundle (or any basic bundle) yields interesting vanishing results. 
\end{remark}

An interesting corollary is {an easy proof in the Riemannian case of the following well-known result for all spin foliations \cite{LMZ}. Notice that the proof given in \cite{LMZ} was completely different and based on the  techniques of sub-Dirac operators}. 

\begin{corollary}\label{LMZresult}
Assume that $F$ is a Riemannian spin foliation of a compact connected oriented manifold $V$, and assume that the $\what{A}$-genus of $V$ is non trivial, i.e. $\what{A} (V)[V] \not = 0$.  Then no compact connected Lie group can act non trivially on $V$ by  leaf-preserving diffeomorphisms.
\end{corollary}

\begin{proof}
 Since $V$ is oriented and the leaves are spin, the transverse bundle $\nu$ is also oriented. {{Since $F$ is  Riemannian, using the Levi-Civita connection on the normal bundle $\nu$ (constructed out of the Bott partial connection), the Pontrjagin classes of the normal bundle $\nu$ are}} represented in $H^*(V;\R)$ by basic closed forms. So, there is a basic closed form $\omega$ with $[\omega]=\what {A} (\nu)$ in $H^*(V;\R)$. If a compact connected Lie group acts non trivially on $V$ and preserves the leaves of the foliation $F$, then replacing this group by a two-fold cover Lie group we may assume that the action preserves the spin structure of the leaves.  See Proposition 2.1 of \cite{AH}.  By Theorem \ref{rigidity} we have,
$$
0   \; = \;   \left<  \int_F \what{A} (TF), C_\omega \right>    \; = \;    \left<  \int_F \what{A} (TF), C_\omega \right>  \; = \;     \int_V \what{A} (TF) \what{A} (\nu)\; = \;   \what{A} (V)[V]  \not = 0,
$$
a contradiction. \end{proof}

Now we prove Proposition \ref{continuity}.

\begin{proof}
Since  every element $h\in H$ acts as a holonomy diffeomorphism,  the multiplier $\Psi^E(h)$ of the algebra $C_c^\infty (\cG, E)$, recalled in Section \ref{Review}, is well defined. Recall that the higher Lefschetz number $L_C(h; E,d)$ of $h$ with respect to the elliptic $H$-invariant complex $(E,d)$ is given by evaluation at $h$ of the $H$-equivariant pairing between the cyclic cocycle $\tau_C$ associated with $C$ and the index class $\ind^H(E,d)\in K^H (C_c^\infty (\cG, E))$. But the equivariant pairing of $\ind^H(E, d)$ with the cyclic cocycle $\tau_C$ was defined in \cite{BH10} as follows. The class $\ind^H(E,d)$ is represented by some formal difference $[\te] - [\te ']$ with $\te=(e, \lambda)$ and $\te'=(e', \lambda')$ being $H$-invariant idempotents in $\widetilde{C_c^\infty (\cG, E)}\otimes \End (X)$  for some finite dimensional representation $X$ of $H$. So here 
$$
\widetilde{C_c^\infty (\cG, E)}\otimes \End (X) = (C_c^\infty (\cG, E)\oplus \C)\otimes \End (X) \simeq \left(C_c^\infty (\cG, E)\otimes \End (X)\right) \oplus \End (X).
$$
Now,  $e, e'\in C_c^\infty (\cG, E)\otimes \End (X)$ while $\lambda, \lambda'\in \End (X)$, satisfy that $\te$ and $\te '$ are  $H$-invariant and  the class $[\te] - [\te ']$ belongs to the kernel  $K^H(C_c^\infty (\cG, E))= \Ker\left(K^H(\widetilde{C_c^\infty (\cG, E))} \rightarrow R(H)\right)$.

\ms

The pairing between $\ind^H(E,d)$ and $\tau_C$, followed by evaluation at $h$,  is then given by the following formula \cite{BH10}
$$
L_C(h;E,d) = (\tau_C\sharp \tr) (\Psi^{E\otimes X} (h)\circ e, e, \cdots , e) - (\tau_C\sharp \tr) (\Psi^{E\otimes X} (h)\circ e', e', \cdots , e'),
$$
where the cyclic cocycle $\tau_C$ is defined in \cite{BH04} as follows using Connes' $X$-trick.
Consider the graded differential algebra $(M_2 (C_c^\infty (\cG, E\otimes \Lambda^\bullet \nu^*), *, \delta)$ with differential 
$$
\delta : M_2 (C_c^\infty (\cG, E\otimes \Lambda^\bullet \nu^*)) \longrightarrow M_2 (C_c^\infty (\cG, E \otimes \Lambda^{\bullet+1}\nu^*)),
$$ 
which is constructed out of the transverse connection and the partial curvature $\theta$, and satisfies $\delta^2=0$.
Now, the cyclic cocycle $\tau_C$ is given as usual by a foliated variation of Connes' formula:
$$
\tau_C (f_0, \cdots, f_{2k}) = \left<  \int_F \tr \left( [(\Psi^E(h)\circ f_0)*\delta f_1 * \cdots *\delta f_{2k}]|_V \right), C\right>.
$$
Since $e$ is smooth with compact support, it is clear that 
$$
(\Psi^E(h)\circ e)*\delta e * \cdots *\delta e,
$$
is a uniformly bounded smoothing operator along the leaves of $\cG$, with coefficents in differential forms on $V$. 
Moreover, it  is also  transversely smooth in the sense of \cite{BH11}[Definition 3.1]. Therefore,  all (transverse) derivatives of $(\Psi^E(h)\circ e)*\delta e * \cdots *\delta e$ yield uniformly bounded smoothing operators. This shows that  the map which assigns to $h$ the differential form on $V$ given by
$$
[(\Psi^E(h)\circ f_0)*\delta f_1 * \cdots *\delta f_{2k}]|_V,
$$
is continuous for the smooth (Fr\'echet) topology of the algebra of differential forms on $V$. Now, the closed holonomy invariant  current $C$ induces a continuous map $\dd C \circ  \int_F$ for this  Fr\'echet topology \cite{Hae}, so the proof is now complete.
\end{proof}

\subsection{Some rigidity Examples}\label{rigidguys}

We now give examples where Theorem \ref{rigidity} can be applied to conclude that there is no action, which are not already covered by {the results of \cite{HL91}}. 
A non-trivial example would be a foliation $F$ on $V$, with the $\what{A}$ genus of  $V$ being zero,  (so that is not an obstruction to an action), and  the $\what{A}$ genus of each leaf is zero (so is not an obstruction to a leafwise action), but there is a closed Haefliger current $C$ of positive degree so that $\dd\left<  \int_F \what{A}(TF), C\right> \neq 0$, and similarly for Corollary \ref{LMZresult} in the general non-Riemannian case considered in \cite{LMZ}.

\ms

Interesting examples already show up with foliations given by compact fiber bundles $F \to V \to N$ where the foliation is the tangent bundle along the fibers $TF$.  They satisfy  $\what{A}(V)[V] = 0$, $\what{A}(F)[F] = 0$, equivalently $\dd \int_F  \what{A}(TF) = 0$ in  $H^0(V/F) = H^0(N;\R)$, but $\dd \int_F  \what{A} (TF) \neq 0$ in $H^*(N;\R)$.  By Poincare duality, there is an element $[\omega] \in H^*(N;\R)$ so that $\dd\left<  \int_F \what{A}(TF), C_{\omega}\right> \neq 0$.

\subsubsection{{Example 1}} Recall Example 7.10 of \cite{BH23}, which is an adaption of Example 1 of \cite{H78}.
\begin{example}\label{SLexample}
Let $G= SL_2\R \times \cdots \times SL_2\R$ (q copies) and $K= SO_2  \times \cdots  \times SO_2$ (q copies). $G$ acts naturally on $\R^{2q}\ssm \{0\}$ and is well known to contain subgroups $\Gam$ with $N = \Gam \backslash G/K$ compact, (in fact a product of $q$ surfaces of higher genus).   Set 
$$
V = \Gam \backslash G \times _K ((\R^{2q} \ssm \{0\})/\Z)  \simeq\Gam \backslash G \times _K (\S^{2q-1} \times \S^1),
$$
 where $n \in \Z$ acts on $\R^{2q} \ssm \{0\}$ by $n\cdot z = e^n z$. 
 
 \ms
 
 $V$ has two transverse foliations, $F$ which is given by the fibers $\S^{2q-1} \times \S^1$ of the fibration $V \to N$, and a transverse foliation coming from the foliation $\tau$ of Example 1 of \cite{H78}.  $\tau$ is defined on the vector bundle $\Gam \backslash G \times _K \R^{2q}$, and the zero section is a leaf of it.  In addition, the action of $\Z$ preserves $\tau$, fixing the zero section, so it descends to a foliation on $V$, also denoted $\tau$.
\end{example}
  
We work with $F$, noting that $TF$ is orientable and spin since $\R^{2q} - \{0\}$ has these structures and the actions of $K$ and $\Z$  preserve them.  The following proposition is proven in the Appendix of \cite{BH23}.

\begin{prop}  \label{volFormOnN}
$\dd \int_F \what{A}(TF)$ is a nowhere zero $2q$ form on $N$.  In particular, there is a non-zero constant $C_q$ so that $\dd \int_N \int_F \what{A}(TF) = C_q \vol(N)$.
\end{prop}

Note that the $\what{A}$ genus of each leaf  $\what{A}(\S^{2q-1} \times \S^1)[(\S^{2q-1} \times \S^1] = 0$, as we wish.  However, 
$$
\what A (V)[V]  =   \int_N \int_F \what{A}(TF)  \what{A}(TN) =   \int_N \int_F \what{A}(TF)  \not = 0,
$$
since $N$ is a product of surfaces, so $\what{A}(TN) = 1$. To overcome this, we take the cross product with $\S^1$and its  zero dimensional foliation $F_0$, that is, set
$$
V_1 = V \times \S^1,  \;\;  N_1 = N \times \S^1,  \;\;  F_1 = F \times F_0 \;\;  \text{and} \;\;  C = d\vol_{\S^1}.
$$
Then this satisfies all our requirements, in particular, the $\what{A}$ genus of each leaf  of $F_1$ is zero, 
$$
\what{A} (V_1)[V_1]  =0, \;\;  \text{and} \;\;  \what{A}(TF_1) = \what{A}(TF),  \;\;  \text{so} 
 \;\;\left< \int_{F_1} \what{A}(TF_1), C\right>  \; = \;  \int_{\S^1}  d\vol_{\S^1}\int_N \int_F \what{A}(TF) 
 \neq 0.
$$
Finally, note that while each leaf does admit actions, their agglomeration does not. 

\begin{remark}\
For further examples, note that the calculations in the examples in \cite{H78} can be used  to provide examples associated to the groups  $G= SL_{2n_1}\R \times \cdots \times SL_{2n_r}\R$, and $K= SO_{2n_1}  \times \cdots  \times SO_{2n_r}$, and $G= SL_{2n_1}\R \times \cdots \times SL_{2n_r}\R \times \R$ and $K= SO_{2n_1}  \times \cdots  \times SO_{2n_r} \times \Z$.  
\end{remark}

\begin{remark}
Note that $\S^1$ can be replaced by any compact orientable manifold $X$ with dimension not congruent to $4$ mod $n$,  with $d\vol_{\S^1}$ replaced by $d\vol_X$.
\end{remark}

\subsubsection{{Example 2}} We get a second example of this type by using the results of Atiyah in \cite{A69}.  In \cite{CHS}, it was shown that the signature class is multiplicative for a fiber bundle with fiber $Y$ and base space $X$, which are compact oriented manifolds,  provided the fundamental group of $X$ acts trivially on the cohomology of $Y$.  Atiyah constructed examples that show this restriction on the action of $\pi_1(X)$ is necessary.   

\ms

In particular, Atiyah constructed  a closed oriented $4$-manifold $Z$  which fibers over a higher genus surface $X$ with fibers a higher genus surface $Y$.   Denote by $TF$ the tangent bundle to the fibers, which is a spin foliation, denoted $F$.  It has the property that its first Chern class $d = c_1(TF)$ satisfies $d^2 \neq 0$ in $H^4(Z,\R)$.  In addition, the total Pontrjagin class of $Z$ is 
 $p(Z) = 1 + d^2$.  Note also that there are spin structures on  $Z$ as well as on $X$ and on the fibers $Y$ over $X$. 
 
\ms

For dimensional reasons,  $\what{A}(Y)[Y] = \what{A}(X)[X] = 0$.
However, since $p(Z) = 1 + d^2$, $\what{A}(Z)[Z] \neq 0$, just as above.  To handle this we again take a product with $\S^1$ with its point foliation $F_0$.  That is, we consider 
$$
Z_1 = Z \times \S^1,  \;\;  X_1 = X \times \S^1,  \;\;  F_1 = F \times F_0 \;\;  \text{and} \;\;  C = d\vol_{\S^1}.
$$
This satisfies all our requirements, in particular, the $\what{A}$ genus of each leaf  of $F_1$ is zero, 
$$
\what{A} (Z_1)[Z_1]  =0, \;\;  \text{and} \;\;  \what{A}(TF_1) = \what{A}(TF),  \;\;  \text{so} 
 \;\;\left< \int_{F_1} \what{A}(TF_1), C\right>  \; = \;  \int_{\S^1}  d\vol_{\S^1}\int_N \int_F \what{A}(TF) 
 \neq 0.
$$

\subsubsection{{Example 3}}

We now give examples using homogeneous spaces.
Recall that {the} complex projective space $\C P_{q}$ is orientable for all $q$, and  is spin if and only if $q$ is odd.  
In addition, 
$$
{\oH^*(\C P_{q}; \Z) = Z[\alpha] / (\alpha^{q+1} = 0), \text{ where }\alpha \in \oH^{2}(\C P_{q}; \Z).}
$$   
It is classical that  the total Pontrjagin class of $\C P_{q}$ is 
$$
p(\C P_{q}) \,\, = \,\, (1 + \alpha^2)^{q+1}, \;\; \text{ so } \;\; 
\what{A}(T\C P_{q}) \,\, = \,\, \Bigl{(}\frac{\alpha/2}{\sinh(\alpha/2)}\Bigr{)}^{q+1},
$$
 and $\dd \what{A}(\C P_{q}) = \int_{\C P_{q}}  \what{A}(T\C P_{q})$, the $\what{A}$-genus of $\C P_{q}$, is the coefficient of $\alpha^{q}$ in this series.  Thus, to get $\what{A}(\C P_{q})$, we need to compute
 $$
\frac{1}{2\pi i} \oint \frac{1}{z^{q+1}}\Bigl{(} \frac{z/2}{\sinh(z/2)}\Bigr{)}^{q+1}  \, dz  \,\, = \,\,
\frac{1}{2\pi i} \oint \frac{1}{2^{q+1}} \frac{1}{(\sinh(z/2))^{q+1}}  \, dz  \,\, = \,\,
2^{-q}\oint     \frac{ (2\pi i)^{-1}  du}{u^{q+1} \sqrt{1+u^2}},
$$
using $u = \sinh(z/2)$.
The integral gives the coefficient of $u^q$ in the Taylor series of $1/\sqrt{1+u^2}$, which equals $q+2$ times the coefficient  of $u^{q+2}$ in the Taylor series of $\sqrt{1+u^2}$.  To see this note that the derivative of $\sqrt{1+u}$  is $1/(2\sqrt{1+u})$, and then substitute $u^2$ for $u$.  We claim that this is non-zero if and only if $q$ is even.  

\ms

Given the claim,  then $\what{A}(\C P_q) =\what{A}(U_{q+1}/(U_1 \times U_q) \neq 0$ if and only if $q$ is even.  Thus we set $ F = \C P_{q+1}$, which is spin with $\what{A}$ genus zero.    Next, consider the universal $\C P_{q+1}$ bundle, 
$$
F = \C P_{q+1} \to EU_{2q+1}/( U_{1} \times U_{2q})\to BU_{2q+1}.
$$
It satisfies $\dd\int_F \what{A}(TF)$ has at least one non-zero term in dimension greater than zero. So there is a compact manifold $N$ and a principal bundle $F = U_{2q+1} \to V \to N$ so that $\dd \int_F \what{A}(TF)$ has at least one non-zero term in dimension greater than zero.  By Poincare duality, there is an element $[\omega] \in H^*(N;\R)$ so that $\dd\left< \int_F \what{A}(TF), C_{\omega}\right> \neq 0$.   If the $\what{A}$ genus of $V$ is not zero, we may take the product with $\S^1$ as above to kill it.   

\ms

The  proof of the claim {is standard and proceeds for instance by induction.}
 We need to show that the Taylor series of $f = \sqrt{1+u^2}$ is an even (which is obvious)  series with non-zero coefficients for all even powers of $u$.   Easy calculations get the induction  started.   So, write $f(u) = \sum_{n=0}^\infty a_n u^{2n}$ and note that $f^2 = 1+u^2$.  Thus we have that for $n \geq 2$
$$
2a_n + \sum_{j=1}^{n-1} a_j a_{n-j} = 0.
$$
The induction hypothesis is that for $j < n$, $a_j = (-1)^{j+1} |a_j| \neq 0$.  Then 
$$
2a_n  \,\, = \,\, - \sum_{j=1}^{n-1} a_j a_{n-j}    \,\, = \,\, 
- \sum_{j=1}^{n-1} (-1)^{j+1}  (-1)^{n-j+1}|a_j| \, | a_{n-j}|
 \,\, = \,\, 
(-1)^{n+1}  \sum_{j=1}^{n-1} |a_j| \, | a_{n-j}| \neq 0 .
$$
Thus $a_n \neq 0$ for all $n$.  Note that the proof also shows that the denominator of $a_n$ is a power of $2$, as required by Theorem 25.4 of \cite{BoH59}.

\subsubsection{{Example 4}} Similar examples can be constructed following Chapter V, Section 15 of \cite{BoH58}.  In particular, denote by $Sp_q$ the space of unitary quaternionic $q \times q$ matrices, and set $\K P_{q-1}  = Sp_q/(Sp_1 \times Sp_{q-1})$.   
Recall that $\oH^*(\K P_{q-1}; \Z) = Z[\alpha] / \alpha^q = 0$, where $\alpha \in \oH^{4}(\K P_{q-1}; \Z)$.   Thus $w_1(\K P_{q-1}) = w_2(\K P_{q-1}) = 0$, so $\K P_{q-1}$ is orientable and spin.
From \cite{BoH59}, p.\ 519, we have that the total Pontrjagin class of $\K P_{q-1}$ is
$$
p(\K P_{q-1}) \,\, = \,\, (1 + 4 \alpha)^{-1}(1 + \alpha)^{2q}, \;\; \text{so} \;\; 
\what{A}(T\K P_{q-1}) \,\, = \,\, \frac{\sinh(\sqrt{\alpha})}{\sqrt{\alpha}} \Bigl{(}\frac{\sqrt{\alpha}/2}{\sinh(\sqrt{\alpha}/2)}\Bigr{)}^{2q} ,
$$
and $\what{A}(\K P_{q-1})$, the $\what{A}$-genus of $\K P_{q-1}$, is the coefficient of $\alpha^{q-1}$ of this series.  If we set $z = \sqrt{\alpha}$, then we need to compute 
$$
\frac{1}{2\pi i} \oint \frac{1}{z^{2q-1}}\frac{\sinh(z)}{z}\Bigl{(} \frac{z/2}{\sinh(z/2)}\Bigr{)}^{2q}  \, dz
 \,\, = \,\,  \frac{1}{2\pi i} \oint \frac{1}{2^{2q}} \frac{\sinh(z)}{(\sinh(z/2))^{2q}} \, dz  \,\, = \,\,
$$
$$
 \frac{1}{2\pi i} \oint \frac{1}{2^{2q}} \frac{2\sinh(z/2)\cosh(z/2))}{(\sinh(z/2))^{2q}} \, dz  \,\, = \,\,
 \frac{1}{2\pi i} \oint \frac{1}{2^{2q-2}} \frac{du}{u^{2q-1}} \, dz  \,\, = \,\, 0,
$$
since $ q > 1$.  Thus the 
$$
\what{A}(\K P_{q-1}) \,\, = \,\, \what{A}(Sp_q/(Sp_1 \times Sp_{q-1}) \,\, = \,\, 0.
$$
We leave it to the reader to continue as above using results of  \cite{BoH58}.  

\subsection{Other rigidity results}

As noted in \cite{HL91}, the results of \cite{BT}, originally conjectured by Witten, extend to foliated manifolds in much the same way that the result of \cite{AH} for the $\what{A}$ genus was extended here.  In particular, the analog of the Bott-Taubes theorem is as follows.

\ms

Denote by $S^k(F)$ and $\Lam^k(F)$ the symmetric and exterior powers of $TF\otimes \C$, and set 
$$
S_a(F) = \sum_{k=0}^{\infty} a^kS^k(F)  \;\; \text{ and } \;\;  \Lam_a(F) = \sum_{k=0}^{\infty} a^k\Lam^k(F).  
 $$
Denote by $\Delta^+$ and $\Delta^-$ the spin representations of $\Spin(2p)$.  Then the signature operator $d_S = \Dir \otimes \Delta_1$ where $\Dir$ is the Dirac operator on a spin manifold of dimension $2p$ and $\Delta_1 = \Delta^+ + \Delta^-$.  An elliptic operator $D$ which is preserved by an $\S^1$ action on a manifold is said to be universally rigid if the induced action of $\S^1$ on  $\ind(D)= \ker(D) -\coker(D)$ is the identity map.

\begin{theorem}Let $F$ be a $2p$ dimensional spin foliation of $V$ with the leafwise Dirac operator $\Dir$, which satisfies the hypothesis of Theorem \ref{basic}, with $H = \S^1$.  Let $R_n$ and $R'_n$ be the defined by the sequences 
$$
R_q \:= \sum_{n=0}^{\infty} q^n R_n = \bigotimes_{n=1}^{\infty} \Lam_{q^n}(F) \bigotimes_{m=1}^{\infty} S_{q^m}(F)
 \;\; \text{ and } \;\; 
 R'_q \:= \sum_{n=0}^{\infty} q^{n/2} R'_n =  \hspace{-0.47cm} \bigotimes_{n=1/2,3/2,...}^{\infty} \hspace{-0.47cm}\Lam_{q^n}(F) \bigotimes_{m=1}^{\infty} S_{q^m}(F). 
$$
Then the leafwise operators $ =\Dir \otimes R'_n$ and $d_S \otimes R_n$ 
acting on sections of the bundles denoted $E'$ and $E$,  respectively,
are universally rigid.  In particular, for all $h \in \S^1$,  see Theorem 5.3 \cite{BH10},  
$$
L_C(h;E', \Dir \otimes R'_n) = L_C(\Id;E', \Dir \otimes R'_n) =
\left< \ch(\sigma(E',\Dir \otimes R'_n))\Td(TF\otimes \C), C \right > = 
\left< \what{A}(TF)\ch(R'_n), C  \right>,
$$
and similarly
$$
L_C(h;E, d_S\otimes R_n)  = \left< L(TF)\ch(R_n), C \right >.
$$
\end{theorem}

\section{Spectral triples and integrality}\label{spec&Integrality}

In this section we explain the construction of transverse spectral triples as proposed by Connes \cite{C86}, and their index theory as  developped by Connes and Moscovici in \cite{CM95}, with the consequences for the higher Lefschetz formulae for foliations. 

\subsection{Transverse spectral triples} 

We shall assume for simplicity that $F$ is a Riemannian foliation, and so can employ the more precise results obtained by Kordyukov in \cite{Kordyukov, Ko09}. For other interesting explicit results on Lefschetz formulae for foliations, see \cite{Ko15}.  We may and will assume that the metric on $V$ is  bundle like, so its restriction to the transverse bundle $\nu=TV/TF$  is holonomy invariant.   

\ms

We shall be using some elements of the basic cohomology  $\oH^{*}_{bas}(V,F)$.  
In particular, let ${\what{E}}$ be a  basic  Hermitian bundle on $V$, which means that there exists a connection $\nabla^{\what{E}}$ on ${\what{E}}$ which is locally projectable, so its curvature $\Ome^{\what{E}}$ is basic.   {Examples of such a bundle are provided by   bundles {which are} functorially constructed out of the transverse bundle $\nu=TV/TF$.  Then all the characteristic classes of $\what{E}$ can be represented by closed basic forms, i.e.  the Chern classes $c_j({\what{E}})$ in the complex case and the Pontrjagin classes $p_j({\what{E}})$ in the real oriented case, all live in $H^{ev}_{bas} (V, F)$.
In fact, $\oH^*_{bas}(V,F)$ contains the ring generated by the Pontrjagin classes $p_j(\nu)$ of $\nu$ and the Chern classes $c_j({\what{E}})$ of all basic Hermitian bundles.}
We will be particularly interested in the Chern character 
$\ch(\what{E}) = [\tr(\exp(\Ome^{\what{E}} /2\pi i))]$ of {the basic hermitian bundle ${\what{E}}$, as well as in  the $\what{A}$ genus $\what{A}(\nu)$,  of the transverse bundle $\nu$.  Recall that the latter is given as follows.}
Denote by $R^\nu$ the curvature of the Levi-Civita connection {$\nabla^{\nu}$} on $\nu${, which is constructed using the Bott connection and is locally the pull-back of the Levi-Civita connection on a transversal}. Then $\what{A}(\nu)$  is the basic cohomology class {represented} by the {closed} basic form $\det  \left(\frac{R^\nu/4\pi}{\sinh(R^\nu/4\pi)}\right)^{1/2}$, i.e.
$$
\what{A} (R^\nu) := \left[\det  \left(\frac{R^\nu/4\pi}{\sinh(R^\nu/4\pi)}\right)^{1/2}\right].
$$

\ms

We now recall the notion of a finite dimensional spectral triple, and  the transverse spectral triples associated with   Riemannian foliations as described and studied in \cite{Kordyukov}. A  (finite dimensional) spectral triple is composed of an (involutive) algebra $\maA$ which is  represented in a separable Hilbert space $\maH$, so $\pi:\maA\to \maL(\maH)$ is an involutive morphism of algebras, and of a (``first order'' unbounded) self-adjoint operator $D$ which interacts properly with the representation. More precisely,  it is assumed that:
\begin{itemize}
\item $D: dom (D)_{\subset \maH}\rightarrow \maH$ is a (densely defined) self-adjoint  operator on $\maH$;
\item For any $a\in \maA$, the operator $\pi (a)$ preserves  $dom (D)$ and the commutator $[D, \pi (a)]$ extends to a bounded operator on $\maH$;
\item (finite dimensionality) There exists $\alpha\in [1, +\infty)$ such that for any $a\in \maA$, the operator $\pi (a) (D^2+I)^{-\alpha/2}$ belongs to the Dixmier ideal $L^{\alpha,\infty} (\maH)$.
\end{itemize}
The least  real number $\alpha$   satisfying the third item, denoted $d$,  is called  the dimension of the spectral triple. So for any complex number $z$ such that $\Re (z) >d$, the operator $\pi (a) (D^2+I)^{-z/2}$ is a trace class operator. In fact, there is a more accurate notion of dimension spectrum for spectral  triples which was introduced in \cite{CM95}. More precisely, we assume furthermore that there exists a discrete subset $\Gamma$  of the complex plane, the dimension spectrum, whose projection to the reals is also discrete, and such that for any $a$  in the  algebra generated by $\maA$ and its commutators with $D$, the holomorphic function  $z\mapsto \Tr\left(\pi (a) (D^2+I)^{-z/2}\right)$ on $\{\Re(z)>d\}$ extends to a meromorphic function on the complex plane whose poles are all in $\Gamma$. The dimension spectrum $\Gamma$ is a simple dimension spectrum when all its elements are at most simple poles.  This reflects roughly the usual notion for closed smooth manifolds with many connected components of different dimensions.  In that case $\Gamma\subset \{n\in \N, n\leq d\}$ where $d$ is the maximal dimension of the connected components, and it is a simple dimension spectrum. When the Hilbert space $\maH$ is moreover endowed with a $\Z_2$-grading, say an involution $\gamma$, which commutes with the representation $\pi$ and anticommutes with $D$, such a spectral triple is called an even spectral triple. See again \cite{C94}  for more details on spectral triples and their associated Fredholm modules and index pairings. 

\ms

The main example of a spectral triple we consider here is the transverse spectral triple associated with the Riemannian foliation $(V, F)$ in \cite{Kordyukov}.  More precisely, take for $\maA$ the smooth convolution algebra $\maA=C_c^\infty (\maG)$ associated with the holonomy groupoid $\maG$. The convolution product is given by
$$
(k_1 k_2) (\gamma) := \int_{\maG^{r(\gamma)}} k_1 (\alpha) k_2 (\alpha^{-1}\gamma)\,  d\eta^{r(\gamma)} (\alpha), \quad k_1, k_2\in \maA\text{ and } \gamma\in \maG,
$$
where $d\eta = d\eta^x$ is the {$\maG$-invariant} Haar system on $s:\maG \to V$, which is defined as the pull-back under the covering map $s$ of the Lebesgue class measure on the leaves of $F$ associated with the fixed metric \cite{C79}. The  involution is as usual defined as $
k^* (\gamma) := \overline{k(\gamma^{-1})}$ for $k\in \maA$ and $\gamma\in \maG.$

\ms

If $E$ is any $\maG$-equivariant (always assumed basic here) hermitian bundle over $V$ (with basic connection $\nabla^E$), then the Hilbert space $L^2 (V, E)$ is endowed with the involutive average representation $\pi_E$ given by
$$
\pi_{E} (k)(\xi) (x)  := \int_{\maG^x} k(\gamma) h_{\gamma} (\xi (s(\gamma))) \, d\eta^x(\gamma), \quad \text{for }k\in \maA, \xi \in \maH\text{ and } x\in V.
$$
Here $h_\gamma: E_{s(\gamma)} \to E_{r(\gamma)}$ is the holonomy action on $E$, corresponding to the element $\gamma\in \maG$.

\ms

Then according to \cite{Kordyukov}, any first order transversally elliptic self-adjoint (pseudo)differential operator acting on the smooth sections of $E$, whose square has a scalar principal symbol allows us to define a spectral triple.   A typical example is a transverse Dirac-type operator, for instance the transverse signature operator when the Riemannian foliation is transversely oriented. More precisely, using the fixed metric on $V$, we consider the orthogonal supplementary bundle of $TF$  as a realization of the transverse bundle $\nu$. This allows us to define  a transverse de Rham differential $d_\nu$ and also its adjoint (over $V$) which is $d_\nu^*$, so that the operator $D=d_\nu+d_\nu^*$ is a transversally elliptic operator whose square has a scalar principal symbol given by  the metric. Together with the   $\Z_2$-grading  of $\Lambda^*_\C \nu^*$ associated with the metric and the fixed transverse  orientation, one obtains  an even spectral triple. This example was expanded in \cite{Kordyukov}.

\ms

We shall assume for simplicity that the even transverse bundle $\nu$ has a holonomy invariant spin structure. This is not necessary but it will simplify some computations of the fixed point formulae below. Then the above transverse signature operator can be recast as a twisted transverse Dirac operator. More generally, we may then consider all the  twisted  transverse Dirac operators obtained using basic bundles.  Fix again  the basic hermitian bundle $\what{E}$ with its basic connection $\nabla^{\what{E}}$ and denote by $\maS=\maS^+\oplus \maS^-$ the spin {super}bundle associated with the normal bundle $\nu$. The Hilbert space is then the space of $L^2$-sections of the basic bundle $E=\maS\otimes {\what{E}}$ over $V$, that is $\maH=L^2 (V, \maS\otimes {\what{E}})$ with its $\Z_2$-grading inherited from the $\Z_2$-grading of $\maS$. The involutive representation  $\pi_{\maS\otimes \what{E}}$ of $\maA$ is then given by the same formula as above and will rather  be denoted by $\pi_{\what{E}}$ once the spin structure has been fixed. 
Next,  consider,  as for the signature operator, the self-adjoint transverse spin-Dirac operator $D=\Dir^{\what{E}}$ with coefficients in $\what{E}$, see \cite{LM89} and also \cite{BH18} for the foliation case. 
Note that the bundle $\maS$ inherits from $\nabla^\nu$ a spin connection $\nabla^{\maS}$ which respects the $\Z_2$-grading. The $\Z_2$-graded connection on $\maS\otimes \what{E}$ which is the tensor product connection of $\nabla^{\maS}$ with the basic connection $\nabla^{\what{E}}$ on ${\what{E}}$, will be denoted $\nabla$. Choose a local orthonormal basis $f_1,...,f_q$   of $\nu^*$ with dual orthonormal basis  $e_1,...,e_q$ of $\nu$.  For $u \in C^\infty (V, \maS\otimes\what{E})$, set 
$$
\Dir_0^{\what{E}}(u) \,\, = \,\, \sum_{1\leq i\leq q} f_i \cdot \nabla_{e_i}u,
$$
where $f_i \cdot$  is the operator $c(f_i)\otimes \id_{\what{E}}$, that is Clifford multiplication by $f_i$.  In general, $\Dir_0^{\what{E}}$ is not self-adjoint.  The mean curvature vector field of $F$   is 
$\mu= \sum_{j=1}^p p_{\nu}(\nabla^{TV}_{X_j} X_j)$ where $X_1,...,X_p$ is a local orthonormal framing of $TF$, $\nabla^{TV}$ is the Levi-Civita connection on $V$, and $p_{\nu}:TV \to \nu$ is the projection.  When we think of $\mu$ as a covector (the isomorphism $\nu \simeq \nu^*$ being given by the inner product), then we denote Clifford multiplication by it acting on $\maS\otimes \what{E}$ by $\mu\cdot = c(\mu)\otimes \id_{\what{E}}$, and it is given by
$$
\mu\cdot \,\, = \,\,  \sum_{j=1}^p  \sum_{i=1}^q \langle  [e_i,X_j],X_j \rangle f_i\cdot
$$
The transverse Dirac operator $\Dir^{\what{E}}$ associated to $F$ is now the self-adjoint operator
$$
\Dir^{\what{E}} \,\, = \,\,  \Dir_0^{\what{E}} \,\, - \,\, \frac{1}{2}\mu\cdot 
$$
It is easy to check that $\Dir^{\what{E}}$ anticommutes with the $\Z_2$-grading.  See \cite{Kordyukov, GlK91}. The following is proven in \cite{Kordyukov}.

\begin{theorem}\cite{Kordyukov}\
The triple $(\maA, \maH, \Dir^{\what{E}})$ is an even spectral triple with dimension equal to $q=\codim (F)$ and with simple dimension spectrum contained in $\{m\in \N, m\leq q\}$.\end{theorem}

\subsection{The equivariant Connes bicomplex}
We shall  represent the equivariant Connes-Chern character of our spectral triple by the corresponding Connes-Moscovici residue cocycle  in the $H$-equivariant cyclic bicomplex $(b_H,B_H)$. Recall that if $H$ is any  compact group which acts continuously on a locally convex (unital) algebra $\maA$, with $h 1=1$ for any $h\in \maA$, then  the equivariant  Hochschild complex $(C^*(\maA, H), b_H)$  is defined as follows \cite{BH10}. The space of cochains $C^n(\maA, H)$ is composed of the continuous functions
$f: \maA^{\otimes_{n+1}} \to C (H)$ such that 
$$
f(h a^0, \cdots, ha^n) (h g h^{-1}) = f(a^0, \cdots, a^n)(g), \quad  \forall g,h\in H, \forall a^j\in \maA .
$$
We denote by $f(a^0, \cdots, a^n|h)$ the scalar $f(a^0, \cdots, a^n)(h)$ for $f\in C^n(\maA, H)$. The equivariant Hochschild differential $b^H: C^n(\maA, H) \to C^{n+1} (\maA, H)$ is then defined for $f\in C^n(\maA, H)$ by
$$
(b_Hf) (a^0, \cdots, a^{n+1} | h):= (b' f)(a^0, \cdots, a^{n+1} | h) + (-1)^{n+1} f ( h^{-1} (a^{n+1}) a^0  , a^1, \cdots, a^n | h).
$$
Here the operator $b'$ is the standard one given by
$$
(b' f)(a^0, \cdots, a^{n+1} | h) := \sum_{j=0}^n (-1)^j f( a^0, \cdots, a^j a^{j+1}, \cdots, a^{n+1}| h).
$$
The relations $b'^2=0$ and $(b_H)^2=0$ are then satisfied. 

\ms

In order to define the equivariant cyclic complex as well as the equivariant cyclic bicomplex, we need to introduce an equivariant version of the Connes operator $B$.  As in \cite{BH10}, denote  by $\lambda_H: C^n(\maA, H)\to C^n(\maA, H)$ the equivariant  cyclic permutation given by
$$
\lambda_H (a^0, \cdots, a^n |h) := (h^{-1} (a^n), a^0, \cdots, a^{n-1}  |h).
$$
Exactly as in the non-equivariant case \cite{C85}, one easily shows that $(\lambda_H)^{n+1}=\id_{C^n(\maA, H)}$. Hence, the equivariant cyclic antisymm\'etrisation operator $A_H: C^n(\maA, H)\to C^n(\maA, H)$ can be defined by the usual formula
$$
A_H = \sum_{j=0}^n (-1)^{nj} (\lambda_H)^j.
$$
Recall from \cite{BH10} the relation $b_H \circ A_H = A_H \circ b'$. Now the equivariant Connes operator, denoted here  $B_H$, will be the operator $B_H:=A_H \circ B_0: C^n (\maA, H) \to C^{n-1} (\maA, H)$ where 
$$
(B_0 f) (a^0, \cdots , a^{n-1} |h) = f(1, a^0, \cdots, a^{n-1} |h) - (-1)^n f(a^0, \cdots, a^{n-1}, 1 |h).  
$$

\begin{lemma}\
The following relations hold  on $C^*(\maA, H)$:
\begin{enumerate}
\item $B_0\circ b_H+ b'\circ B_0 = \id +(-1)^{n+1}\lambda_H$ on $n$-cochains.
\item $b_H\circ B_H + B_H \circ b_H =0$.
\end{enumerate}
\end{lemma}

\begin{proof}\
The first item is a straightforward computation. If we compose this relation on the left with the operator $A_H$, then we get on $n$-cochains
$$
B_H \circ b_H + A_H\circ b'\circ B_0 = A_H + (-1)^{n+1} A_H\circ \lambda_H.
$$
But $A_H\circ b'= b_H\circ A_H$ and hence $A_H\circ b'\circ B_0=b_H\circ B_H$. Moreover, on $C^n(\maA, H)$, we have
$$
A_H\circ \lambda_H= (-1)^n A_H\text{ therefore }(-1)^{n+1} A_H\circ \lambda_H = - A_H.
$$
So, we finally get the second relation $B_H\circ b_H + b_H\circ B_H =0$ on $C^*(\maA, H)$. 
\end{proof}

As in the non-equivariant case, the equivariant cyclic cohomology of the algebra $\maA$, as defined in \cite{BH10} using the equivariant cyclic subcomplex of the equivariant Hochschild complex, can be recovered from  the second  filtration of the bicomplex $(b_H, B_H)$.  See \cite{C94} for more precise details. 

\subsection{The equivariant residue cocycle} 
Next we take into account the leafwise action of the group $H$ which is again supposed to act by holonomy diffeomorphisms. Our algebra $\maA= C_c^\infty (\maG)$ is not unital, and the unit appearing in the above formula for $B_0$ can be an added $H$-trivial unit.   We denote again by $\gamma$ the grading involution of $\maH$ and by $U(h)$ the unitary of the Hilbert space  $\maH$ which is associated with an element $h\in H$. The following is a straightforward generalization to the equivariant case of the Connes-Moscovici local index theorem \cite{CM95}. See also \cite{BH18, BH19}.

\begin{theorem}\label{EquivCocycle}\
For $a_0\in \maA$ and $h\in H$, set $
\phi_0(a_0 \vert h) \,\, = \,\, \Res_{z=0} \Bigl[ z^{-1}\, \Tr \left( \gamma U(h) a_0(\id + (\Dir^{\what{E}})^2)^{-z}\right)  \Bigr].$   In addition, for 
$k >0$, $k$ even, and for $a_j\in \maA$, set
$$
\phi_k (a_0, \cdots , a_k \vert h)  \,\, := \,\, \frac{(\frac{k}{2}-1)!}{k!} \, \Res_{z=0} \left[\Tr \left(\gamma \, U(h) \, a^0 [\Dir^{\what{E}}, a^1] \cdots [\Dir^{\what{E}}, a^k] \, (\id + (\Dir^{\what{E}})^2)^{-k/2 - z}\right)\right].
$$
Then for any even $k\geq 0$, $\phi_k\in C^k(\maA, H)$ (with $\phi_k=0$ for $k>\codim (F)$), and $b_H \phi_k + B_H\phi_{k+2} = 0$.  Hence, $\phi=(\phi_k)_{k}$ is a finitely supported  equivariant cyclic cocycle over the $H$-algebra $\maA$.
\end{theorem}
 
The representation $\pi_{\what{E}}$ is of course implicit in the formulae for $\phi_k$. 

 \begin{proof}\  We first point out that each $\phi_k$ belongs to $C^k (\maA, H)$. Indeed, the unitary $U(h)$ commutes with the Dirac operator $\Dir^{\what{E}}$ and the representation of $C_c^\infty (\maG)$ in the Hilbert space $\maH$ is $H$-equivariant, therefore,
\begin{eqnarray*}
 \phi_0 (ha_0 \vert hgh^{-1}) &=& \Res_{z=0} \Bigl[ z^{-1}\, \Tr \left( \gamma U(hgh^{-1}) (ha_0) (\id + (\Dir^{\what{E}})^2)^{-z}\right)  \Bigr]\\
 & = & \Res_{z=0} \Bigl[ z^{-1}\, \Tr \left( \gamma U(h) U(g)  a_0 U(h^{-1}) (\id + (\Dir^{\what{E}})^2)^{-z}\right)  \Bigr]\\
 & = & \Res_{z=0} \Bigl[ z^{-1}\, \Tr \left(U(h)  \gamma U(g) a_0  (\id + (\Dir^{\what{E}})^2)^{-z} U(h^{-1}) \right)  \Bigr]\\
 & = & \Res_{z=0} \Bigl[ z^{-1}\, \Tr \left( \gamma  U(g) a_0  (\id + (\Dir^{\what{E}})^2)^{-z}) \right)  \Bigr]\\
 & = &  \phi_0 (a_0 \vert g).
 \end{eqnarray*}
 In the same way and using again the $H$-invariance of $\Dir^{\what{E}}$ and the $H$-equivariance of the representation, we get
 \begin{eqnarray*}
 &  &\Tr \left(\gamma \, U(hgh^{-1}) \, (ha^0) [\Dir^{\what{E}}, (ha^1)] \cdots [\Dir^{\what{E}}, (ha^k)] \, (\id + (\Dir^{\what{E}})^2)^{-k/2 - z}\right)\\
  & = &\Tr \left(\gamma \, U(h) U(g) \, (a^0 U(h^{-1}) [\Dir^{\what{E}}, U(h) a^1 U(h^{-1}))] \cdots [\Dir^{\what{E}}, U(h) a^k U(h^{-1}))] \, (\id + (\Dir^{\what{E}})^2)^{-k/2 - z}\right)\\
  &=& \Tr \left(U(h) \gamma \, U(g) \, (a^0 U(h^{-1}) U(h) [\Dir^{\what{E}}, a^1] U(h^{-1})) \cdots U(h) [\Dir^{\what{E}}, a^k] U(h^{-1})) \, (\id + (\Dir^{\what{E}})^2)^{-k/2 - z}\right)\\
  &=& \Tr \left(U(h) \gamma \, U(g) a^0  [\Dir^{\what{E}}, a^1] \cdots  [\Dir^{\what{E}}, a^k] \, (\id + (\Dir^{\what{E}})^2)^{-k/2 - z}U(h^{-1})) \right)\\
  &=& \Tr \left( \gamma \, U(g) a^0  [\Dir^{\what{E}}, a^1] \cdots  [\Dir^{\what{E}}, a^k] \, (\id + (\Dir^{\what{E}})^2)^{-k/2 - z} \right).\\
 \end{eqnarray*}
 Hence, we also have the equivariance property for $k>0$
 $$
 \phi_k (ha_0, \cdots , ha_k \vert hgh^{-1}) = \phi_k (a_0, \cdots , a_k \vert g).
 $$
 Since the spectral triple has dimension equal to $\codim (F)$, we know that for $k>\codim (F)$, the operator $(\id + (\Dir^{\what{E}})^2)^{-k/2}$ is trace-class and hence the corresponding residue vanishes. This shows that $\phi_k=0$ whenever $k>q$. 
 
 \ms
 
Since the group $H$ acts by holonomy diffeomorphisms, the multiplier $\psi (h)$ of the algebra $C_c^\infty (\maG)$, as considered in \cite{BH10}, is well defined. More precisely, using the smooth section $\varphi^h$, the multiplier $\psi (h)$ is defined as
$$
\psi (h) (\xi) (\gamma) := \xi ((h^{-1}\gamma) \circ \varphi^{h^{-1}}(s(\gamma))), \quad \text{ for }\xi\in C_c^\infty (\maG), \gamma\in \maG.
$$
Moreover, we have for any $a\in C_c^\infty (\maG)$
 $$
 \pi_{\what{E}} (\psi (h) a) = U(h) \circ \pi_{\what{E}} (a).
 $$
 See \cite{BH10}. Therefore, denoting by $\phi_k^{CM}$ the Connes-Moscovici cocycle defined by 
 $$
 \phi_k^{CM} (a_0, \cdots, a_k) = \phi_k (a_0, \cdots, a_k\vert 1_H),
 $$
 we deduce that for any even $k\geq 0$, the following relation holds
 $$
 \phi_k (a_0, \cdots , a_k \vert h) = \phi_k^{CM} (\psi (h) a_0, a_1, \cdots, a_k). 
 $$
 Note that since the spectral triple has simple spectrum dimension, all the higher residues and commutators with the square of the operator $\Dir^{\what{E}}$ cancell out, and we are reduced to the value $\phi_k (a_0, \cdots, a_k\vert 1_H)$ for the Connes-Moscovici cocycle. 
 
\ms

Computing $b_H\phi_k$ we get 
$$
(b_H\phi_k) (a_0, \cdots, a_{k+1}\vert h) = (b\phi_k^{CM}) (\psi(h) a_0, a_1, \cdots, a_{k+1}).
$$
In the same way we have for any $a\in C_c^\infty(\maG)$, $U(h)[\Dir^{\what{E}},a] = [\Dir^{\what{E}}, \psi (h) a]$ and hence
 $$
 B_0\phi_{k+2} (a_0, \cdots, a_{k+1}\vert h) = B_0\phi_{k+2}^{CM} (\psi(h) a_0, a_1, \cdots, a_{k+1}).
 $$
 Thus we have
 \begin{eqnarray*}
 \lambda_H (B_0\phi_{k+2}) (a_0, \cdots, a_{k+1}\vert h) &=& (B_0\phi_{k+2}) (h^{-1}a_{k+1}, a_0, \cdots, a_k\vert h)\\
 &=& B_0\phi_{k+2}^{CM} (\psi(h) (h^{-1} a_{k+1}), a_0, a_1, \cdots, a_{k})\\
 &=& B_0\phi_{k+2}^{CM} (a_{k+1}\psi(h), a_0, a_1, \cdots, a_{k})\\
 &=& \phi_{k+2}^{CM} (1, a_{k+1}\psi(h), a_0, \cdots, a_k).
 \end{eqnarray*}
 Using again the $H$-equivariance of the representation and the $H$-invariance of the operator $\Dir^{\what{E}}$, we easily see that
 $$
 \phi_{k+2}^{CM} (1, a_{k+1}\psi(h), a_0, \cdots, a_k) =\lambda (B_0\phi_{k+2}^{CM}) (\psi(h) a_0, a_1, \cdots, a_{k+1}).
 $$
 Therefore,
$$
 B_H\phi_{k+2} (a_0, \cdots, a_{k+2}\vert h)  = B\phi_{k+2}^{CM} (\psi(h) a_0, \cdots, a_{k+1}),
 $$
 and we thus deduce 
 $$
 (b_H\phi_k + B_H \phi_{k+2}) (a_0, \cdots, a_{k+1}\vert h) = (b\phi_k^{CM} + B\phi_{k+2}^{CM}) (\psi(h) a_0, a_1,  \cdots, a_{k+1}) = 0.
 $$
 The last equality is due to the cocycle relation satisfied in the non-equivariant setting by the Connes-Moscovici cochain, see again \cite{CM95}. 
 \end{proof}

\begin{remark}\label{StrongEquivariance}
The main property of the  cochains $\phi^{CM}_k$ used in this proof of Theorem \ref{EquivCocycle} is that they are equivariant, due to the stronger  relations:
$$
\phi^{CM}_k (a_0, \cdots, a_{j-1}, a_j \psi(h), a_{j+1}, \cdots, a_k) = \phi^{CM}_k (a_0, \cdots, a_{j-1}, a_j,  \psi(h) a_{j+1}, \cdots, a_k)\text{ for }0\leq j\leq k-1,
$$
plus the relation $\phi^{CM}_k (a_0, \cdots, a_{k-1}, a_k \psi(h)) = \phi^{CM}_k (\psi(h) a_0, a_1, \cdots, a_k)$. 
\end{remark}

Theorem \ref{EquivCocycle} shows that the $(b_H, B_H)$-cocycle $(\phi_k)$ defines a class $[\phi]$ in the (entire) equivariant cyclic cohomology of the algebra $C_c^\infty (\maG)$. Therefore, it can be paired with the equivariant $K$-theory to yield a central function on $H$. 
   
\begin{theorem}\label{Integrality}\
Assume that $h$ topologically generates a compact Lie group denonted $H$. If  $e$ is an $H$-invariant idempotent in $C_c^\infty (\maG)\otimes \End (X)$ for some finite dimension $H$-representation $\rho:H\to \End(X)$, then the following combination of  residues of zeta functions of our operators:
$$
\sum_{n=0}^{\codim(F)/2} \frac{(-1)^n (2n)!}{n!} (\phi_{2n}\sharp \Tr) ((\psi (h)\otimes \rho(h))(e-1/2), e, \cdots, e),
$$
depends only on the equivariant $K$-theory class of $e$, and moreover it  belongs to $R(H) (h)=\{\chi(h), \chi\in R(H)\}$. 
\end{theorem}

\begin{remark}
If for instance $h$ has finite order $\kappa$, then $R(H)(h)=\Z[e^{2i\pi/\kappa}]$, so for a holonomy diffeomorphism which is an involution we always get an integer.
\end{remark}
 
\begin{proof}\
Assume first that the operator $\Dir^{\what{E}}$ is invertible and denote by  $\tau$ the equivariant cyclic cocycle of even degree $q'$, for some $q'>q$, which  represents the Connes-Chern character of the even Fredholm module $(\maH, F)$, where $F=\sign (\Dir^{\what{E}})$, a self-adjoint symmetry. So, by its very definition, the equivariant pairing of the class of $\tau$ with the equivariant $K$-theory class represented by $e$ is an equivariant index pairing, i.e.
$$
\langle [e], [\tau]\rangle  (h) = \Ind^H \left( e (F\otimes \id_X)e : e(\maH^+\otimes X)\to e(\maH^-\otimes X)\right) (h),
$$
where $e (F\otimes \id_X)e : e(\maH^+\otimes X)\to e(\maH^-\otimes X)$ is a Fredholm $H$-invariant operator and its equivariant index is the virtual $[\Ker] - [\Coker]$ representation
and hence belongs to $R(H)$.  Moreover, we have the following explicit formula for $\tau$ 
$$
\tau (a_0, \cdots, a_{q'}\vert h) = C(q') \Tr\left( \gamma U (h) a_0 [F, a_1] \cdots [F,  a_{q'}]\right),
$$
An important observation is then that $\tau$ also satisfies the stronger equivariance property described in Remark \ref{StrongEquivariance}, and we have 
 $$
 \tau (a_0, \cdots, a_{q'}\vert h) = \tau (\psi (h) a_0, \cdots, a_{q'}\vert 1)=:\tau^{CM} (\psi(h) a_0, \cdots, a_{q'}),
 $$
so that setting  $\tau^{CM} (a_0, \cdots, a_{q'}):= \tau ( a_0, a_1, \cdots, a_{q'}\vert 1)$ we get  exactly the corresponding cyclic cocycle which represents the non-equivariant Connes-Chern character of the even Fredholm module $(\maH, F)$ as described in \cite{CM95}. 

\ms

One of the main results  proven in \cite{CM95} is that   $\tau^{CM} - \phi^{CM} $ is a $(b+B)$-coboundary, i.e. there exists an odd transgression cochain $\theta^{CM}=(\theta^{CM}_{2j+1})_j$ such that $\tau^{CM} - \phi^{CM} = (b+B) \theta^{CM}$. A carefull inspection of the cochains $\theta^{CM}_{2j+1}$ again shows that they satisfy as well the strong equivariance property of Remark \ref{StrongEquivariance}. Hence, we obtain exactly as in the proof of Theorem \ref{EquivCocycle} the relations
$$
 b\theta^{CM}_{2j-1} (\psi(h) a_0, a_1, \cdots, a_{2j}) = b_H\theta_{2j-1} (a_0, \cdots, a_{2j}\vert h)
 $$
\and 
 $$
 B\theta^{CM}_{2j+1} (\psi(h) a_0, a_1, \cdots, a_{2j}) = B_H\theta_{2j+1} (a_0, \cdots, a_{2j}\vert h),
$$
where $\theta_{2j\pm 1} (a_0, \cdots, a_{2j\pm 1}\vert h) := \theta^{CM}_{2j\pm 1} (\psi (h) a_0, \cdots, a_{2j\pm 1})$. A consequence is that we have equality in the equivariant cyclic cohomology (described by the $(b_H, B_H)$-bicomplex) of the class of $\phi$ and that of $\tau$. Therefore, we finally obtain
$$
\langle [e], [\phi]\rangle  \, \in\; R(H)\quad \text{ and hence } \quad \langle [e], [\phi]\rangle (h)  \, \in\; R(H)(h). 
$$
Now the formula for the pairing $\langle [e], [\phi]\rangle (h)$ is exactly given by \cite{C94,CM95},
$$
\langle [e], [\phi]\rangle (h) = \sum_{n=0}^{\codim(F)/2} \frac{(-1)^n (2n)!}{n!} (\phi_{2n}\sharp \Tr) ((\psi (h)\otimes \rho(h))(e-1/2), e, \cdots, e).
$$
Since  the operator  $\Dir^{\what{E}}$ is not invertible in general (note also that the algebra $C_c^\infty (\maG)$ is rarely unital), we deform $H$-equivariantly the spectral triple as follows. Define 
$$
\what{\maH} := \maH \oplus \maH\text{ with the new grading } \what\gamma := \left(\begin{array}{cc} \gamma & 0 \\ 0 & -\gamma\end{array}\right) \text{ and new operator } {\widehat {\Dir}}^{\what{E}}:= \left(\begin{array}{cc} \Dir^{\what{E}} & Id \\ Id & -\Dir^{\what{E}}\end{array}\right).
$$
The action of $H$ is simply extended to $\widehat\maH$ as the diagonal action, so  $U(h)$ is replaced by ${\widehat U} (h)$ which is given by $\left[\begin{array}{cc} U (h) & 0 \\ 0 & U(h) \end{array}\right]$.  Now with the new representation $\what\pi: k\mapsto \left(\begin{array}{cc} \pi_{\what{E}} (k) & 0 \\ 0 & 0 \end{array}\right)$, it is easy to check that  we get again an even $H$-equivariant spectral triple $(\maA, \widehat\maH, {\widehat {\Dir}}^{\what{E}})$ with invertible operator. Moreover, a straightforward $H$-equivariant homotopy argument shows that the class, in the (entire) $H$-equivariant $(b^H, B^H)$ cyclic bicomplex, of the residue cocycle ${\what\phi}$ associated with $(\maA, \widehat\maH, {\widehat {\Dir}}^{\what{E}})$, say obtained as $\phi$ is but replacing the representation $\pi_{\what{E}}$ by $\what\pi$ and the operator ${\Dir}^{\what{E}}$ by  ${\widehat {\Dir}}^{\what{E}}$,  is unchanged if we consider the path of operators ${\widehat {\Dir}}^{\what{E}}_{\ep}$ for $\ep\in [0,1]$:
$$
{\widehat {\Dir}}^{\what{E}}_{\ep} := \left(\begin{array}{cc} \Dir^{\what{E}} & \ep Id \\ \ep Id & -\Dir^{\what{E}}\end{array}\right).
$$
Now for $\ep = 0$ we have
\begin{multline*}
\Tr \left(\what\gamma\widehat\psi(h)  \what\pi (a_0) [ \widehat\Dir^{\what{E}}_0, \what\pi (a_1)] \cdots [\widehat\Dir^{\what{E}}_0, \what\pi (a_{q'})] (I + ({\widehat {\Dir}}^{\what{E}}_0)^2)^{-q'/2-z}\right) \\= \Tr \left(\gamma  \Psi (h) a_0 [ \Dir^{\what{E}}, a_1] \cdots [\Dir^{\what{E}}, a^{q'}] (I + {\Dir^{\what{E}}}^2)^{-q'/2-z}\right).
\end{multline*}
As a consequence we see that the pairing of the class of $\phi$ with equivariant $K$-theory coincides with that of the class of the equivariant cocycle $\what\phi$. If we now consider as above the even $H$-equivariant Fredolm module $(\what\maH, \what F)$ where $\what F=\sign ({\widehat {\Dir}}^{\what{E}})$ is the self-adjoint symmetry which is the sign of the invertble operator ${\widehat {\Dir}}^{\what{E}}$, then for any even $q'>q$, we get again an equivariant cyclic Hochschild cocycle by setting just as above
$$
\tau (a_0, \cdots, a_{q'}\vert h) := C(q') \Tr \left(\what\gamma {\what U}(h) \what\pi (a_0) [{\what F}, \what\pi (a_1)]\cdots  [{\what F}, \what\pi (a_{q'})]\right),
$$
where $C(q')$ is some appropriate normalisation constant \cite{C86}. 
The proof is now complete. 
 \end{proof}
 
 \begin{remark}\
Another method to overcome the non-invertibility of the operator ${\Dir}^{\what{E}}$ is to first reduce to the unital case and then apply some standard perturbations, see for instance \cite{GorokhovskyThesis}$[\text{Remark } 2.2]$. A more systematic study of non unital Fredholm modules is also  carried out in  \cite{GayralSukochev2012}, where the authors include semi-finite spectral triples over non-unital algebras as well.
\end{remark}  
 
 \begin{remark}
 In the case of non-Riemannian foliations,  in \cite{CM95} Connes and Moscovici gave a general reduction process  to the case of triangular structures by using the Connes fibration. The resulting foliations  were also called almost Riemannian foliations in \cite{BH18}.  The  spectral triple obtained in this more general case produces a highly involved formula for the cyclic cocycle $\phi=(\phi_{k})_k$, see again \cite{CM95}. However, the previous integrality theorem is still valid since again the pairing of the class of $\phi$ with an equivariant $K$-theory class coincides with that of an equivariant Fredholm module and hence with an equivariant index pairing. 
 \end{remark} 
 
\begin{corollary}\label{EqP}
The  pairing,  with $K^H(C_c^\infty (\maG))$, of the equivariant residue class $[\phi]$ associated with the spectral triple $(\maA, \maH, \Dir^{\what{E}})$ extends to a pairing with $K^H(C^*(V, F))$, and takes values in the representation ring $R(H)\subset C(H)$. So, for any topological generator $h$ of $H$ and any $x\in K^H(C^*(V, F))$, we have
$$
<[\phi] , x >   \,\,\,  \in \,\,\,  R(H),  \quad \text{ and hence }\quad <[\phi] , x >   (h) \,\,\,  \in \,\,\,  R(H) (h).
$$
\end{corollary}

\begin{proof}\ 
A general method of domains of derivations of commutators with the operator \cite{C86} allows one to define an $H$ subalgebra $\maB$ of $C^*(V, F)$ such that:
\begin{itemize}
\item $\maB$ contains $\maA=C_c^\infty (\maG)$;
\item $\maB$ is closed under the holomorphic functional calculus;
\item the Connes-Moscovici residue cochains $\phi^{CM}_k$ all extend to $\maB$ and define a $(b+B)$-closed cocycle on $\maB$, given by the same formulae.  
\end{itemize}
Therefore, the cochains $\phi^{CM}_k$ all satisfy again the strong equivariance property of Remark \ref{StrongEquivariance} and allow us to define the equivariant $(b_H+B_H)$-cocycle $\phi$ on $\maB$ as above. The proof of Theorem \ref{Integrality} now shows that for any $H$-invariant idempotent $e$ in $\maB\otimes \End (X)$ rather than in $\maA\otimes \End (X)$, the pairing of $\phi$ with the class of $e$ in $K^H(\maB)$ belongs again to $R(H)$ since it is given by an $H$-equivariant Fredholm index pairing. Since $\maB$ is closed under the holomorphic functional calculus and is an $H$-subalgebra, we know that the inclusion of $\maB$ induces an isomorphism $K^H(\maB)\simeq K^H(C^*(V, F))$ and every idempotent $p\in C^*(V, F)\otimes \End (X)$ can be approximated by an idempotent $e\in \maB\otimes \End (X)$ which defines the same class in $K^H(C^*(V, F))$. 
On the other hand, recall that the Connes $C^*$-algebra $C^*(V, F)$ is a stable $H$-algebra \cite{HilsumSkandalis, BenameurPacific}, and hence its (equivariant) $K$-theory can be represented by idempotents in $C^*(V, F)\otimes \End (X)$ for finite dimensional representations of $H$ in $X$. This shows that  the pairing of $[\phi]$ with $K^H(\maB)$ is completely determined by the pairings with $H$-invariant idempotents in $\maB\otimes \End (X)$. The proof is thus complete. 
\end{proof}

\subsection{Integrality in higher Lefschetz formulae} Before applying these results to the Lefschetz fixed point formula, we mention the following important theorem proven in \cite{BH18}, see also \cite{BH19}.  Recall the HKR chain map constructed in \cite{BH04} between the de Rham complex of holonomy invariant Haefliger currents and the cyclic complex of the algebra $C_c^\infty(\maG)$. It induces:
$$
\chi: H_{ev/odd}(V/F) \longrightarrow H_\lambda^{ev/odd} (C_c^\infty (\maG)).
$$

We are now in position to combine the previous results with the $K$-theory Lefschetz theorem of \cite{B97}. Recall that $H$ is topologically generated by $h$ and that it acts by isometries of $(V, g)$ which are holonomy diffeomorphisms, and moreover, the fixed point submanifold $V^h=V^H$ is assumed to be transverse to the foliation $F$, so it inherits a foliation $F^h$. Again when $H$ is connected, all these conditions are automatically satisfied when $H$ preserves the leaves. All the previous data can then be restricted to the $H$-trivial foliation $(V^h, F^h)$, and we have an $H$-equivariant spectral triple $(\maA^h, \maH^h, \Dir^{\what{E^h}})$ constructed in exactly the same way, or by restricting $(\maA, \maH, \Dir^{\what{E}})$. More precisely,
\begin{itemize}
\item $\maA^h$ denotes the involutive convolution algebra associated with the holonomy groupoid of the foliation $(V^h, F^h)$. 
\item ${\what{E^h}}={\what{E}}\vert_{V^h}$ and $\maH^h$ is the Hilbert space of $L^2$ sections over $V^h$ of the basic bundle $\maS^h\otimes  {\what{E^h}}$ with $\maS^h$ being the spin bundle of the transverse bundle to the foliation $F^h$ or the restriction of the ambiant spin bundle $\maS$ to $V^h$.
\item The operator $\Dir^{\what{E^h}}$ is the transverse Dirac operator on $(V^h, F^h)$ with coefficients in ${\what{E^h}}$. Its principal symbol is the restriction of the principal symbol of ${\what{E}}$ to $TV^h$.
\end{itemize}

The classical integrality theorems involving rational characteristic classes and proven using the Atiyah-Segal-Singer Lefschetz fixed point formulae \cite{ABII, AS2, AS3, HirzebruchZagier}, can in principle be extended to smooth foliations of closed manifolds. The following  theorem is only a first important step towards proving such integrality results, for two reasons. The first is that it is stated for Riemannian foliations where we could go further and deduce the integrality of the resulting rational combination of characteristic classes, in the spirit of the classical formulae. The second reason is that while this result can be stated for more general foliations using the Connes-Moscovici construction \cite{CM95}, it is at present rather hard to deduce the  integrality of corresponding characteristic numbers. These will certainly involve more exotic characteristic numbers and will be dealt with elsewhere. 

\begin{theorem}\label{LefInt}\
Denote by $\phi_{V^h, F^h}^{CM}$ the even Connes-Moscovici residue cocycle in the $(b, B)$-bicomplex  associated with the fixed point  foliation $(V^h, F^h)$ \cite{CM95}. Denote by $\Ind^{CS}_{V^h, F^h}: K(TF^h) \rightarrow K (C^*(V^h, F^h))$ the Connes-Skandalis topological longitudinal index morphism for the foliation $(V^h, F^h)$ \cite{CS84}. Then for any  leafwise elliptic  $H$-invariant pseudodifferential complex $(E,d)$ over the ambiant foliation $(V, F)$, we have 
$$
\left\langle (\Ind^{CS}_{V^h, F^h}\otimes \C)\left(\frac{i^*[\sigma (E,d)] (h)}{\lambda_{-1} (N^h\otimes \C) (h)}\right) \, , \, [\phi_{V^h, F^h}^{CM}]\otimes \id_\C\right\rangle\; \in \; R(H)(h). 
$$
\end{theorem}
 
Some explanations are needed here. Note that the class $\frac{i^*[\sigma (E,d)] }{\lambda_{-1} (N^h\otimes \C)}$ belongs to the localized $R(H)_h$-module $K^H(TF^h)_h\simeq K(TF^h)\otimes R(H)_h$ since $H$ acts trivially on $TF^h$, hence evaluating at $h$ we get an element of $K(TF^h)\otimes \C$, and so 
$$
(\Ind^{CS}_{V^h, F^h}\otimes \C)\left(\frac{i^*[\sigma (E,d)] (h)}{\lambda_{-1} (N^h\otimes \C) (h)}\right) \; \in \; K(C^*(V^h, F^h))\otimes \C.
$$

\begin{proof}
By the $K$-theory Lefschetz fixed point theorem \cite{B97}, the Lefschetz class $L(h; E,d)$ of $h$ with respect to $(E, d)$ coincides with the image of the class 
$$
(\Ind^{CS}_{V^h, F^h}\otimes \id_{R(H)_h})\left(\frac{i^*[\sigma (E,d)] }{\lambda_{-1} (N^h\otimes \C) }\right)
$$
under  $\maM\otimes \id_{R(H)_h}: K(C^*(V^h, F^h))\otimes R(H)_h\rightarrow K(C^*(V, F))\otimes R(H)_h\simeq K^H(C^*(V, F))_h$. Here $\maM$ denotes the Morita extension morphism associated with the transverse submanifold $V^h$ to the foliation $(V, F)$, see \cite{B97}. The important feature of the Connes-Moscovici residue cocycle in this case of a Riemannian foliation, is that it ``commutes with Morita extension''. More precisely, we have 
$$
\langle \maM(y), [\phi^{CM}]\rangle  = \langle y, [\phi_{V^h, F^h}^{CM}]\rangle, \quad \forall y\in K(C^*(V^h, F^h)).
$$
This can be proven directly from the explicit formula for the cochains $\phi_k$ and for the map $\maM$, by adapting the proof given in \cite{BH10} for holonomy invariant currents. Therefore, we deduce from the $K$-theory Lefschetz theorem that
$$
\left\langle (\Ind^{CS}_{V^h, F^h}\otimes R(H)_h)\left(\frac{i^*[\sigma (E,d)] }{\lambda_{-1} (N^h\otimes \C)}\right) \, , \, [\phi_{V^h, F^h}^{CM}]\otimes R(H)_h\right\rangle = \langle L(h; E,d), [\phi^{CM}]\otimes R(H)_h\rangle. 
$$
But 
$$
\langle L(h; E,d), [\phi^{CM}]\otimes R(H)_h\rangle = \langle \Ind^H([\sigma(E,d)]), [\phi^{CM}]\otimes R(H)\rangle.  
$$
So
$$
\langle L(h; E,d), [\phi^{CM}]\otimes R(H)_h\rangle  (h) = \langle \Ind^H([\sigma(E,d)]), [\phi^{CM}]\otimes R(H)\rangle (h)\; \in \; R(H)(h). 
$$
Hence we end up with,
$$
\left\langle (\Ind^{CS}_{V^h, F^h}\otimes \C)\left(\frac{i^*[\sigma (E,d)] (h)}{\lambda_{-1} (N^h\otimes \C) (h)}\right) \, , \, [\phi_{V^h, F^h}^{CM}]\otimes \id_\C\right\rangle 
= 
$$
$$\left\langle (\Ind^{CS}_{V^h, F^h}\otimes R(H)_h)\left(\frac{i^*[\sigma (E,d)] }{\lambda_{-1} (N^h\otimes \C) }\right) \, , \, [\phi_{V^h, F^h}^{CM}]\otimes R(H)_h\right\rangle \; (h)  \; \in \; R(H)(h).
$$
\end{proof}

In this proof, we used  the explicit formula for  the Connes-Moscovici residue cocycle in this simpler case of Riemannian foliations to easily deduce the natural commutation  with the Morita extension map $\maM$. Applying the results of \cite{BH18, BH19}, that is the techniques of the Getzler rescaling argument on foliations as developped there, this commutation property can also be deduced from the Morita  compatibility property proven in \cite{BH10} for cyclic cocycles associated with closed holonomy invariant currents. More precisely, denote by $[C]$ the Haefliger homology class associated with the basic cohomology class $\what{A} (\nu) \ch ({\what{E}})$ and by $[C\vert_{V^h}]$ its restriction to $(V^h, F^h)$ or the Haefliger class in $H_{ev}(V^h/F^h)$ associated with the basic cohomology class $\what{A} (\nu^h) \ch ({\what{E}} \,|\,_{V^h})$. Then the Connes-Moscovici cyclic cohomology class $[\phi_{V^h, F^h}^{CM}]$ belongs to the range under the HKR map $H_*(V^h/F^h)\to H^*_\lambda (\maA^h)$ for the foliation  $(V^h, F^h)$. Indeed,  it is exactly the image of the Haefliger homology class associated with the basic cohomology class $\what{A} (\nu^h) \ch ({\what{E}} \,|\,_{V^h})$.   Thus we have the following topological description of the Connes-Moscovici residue cocycle which was proven in \cite{BH18, BH19}.

\begin{theorem}\label{ECCc}\cite{BH18, BH19}\
For the Riemannian foliation $(V, F)$, the class $[\phi^{CM}]\in HC^{ev} (C_c^\infty (\maG))$ of the Connes-Moscovici residue cocycle associated with the spectral triple $(\maA, \maH, \Dir^{\what{E}})$ belongs to the range of the HKR map $\chi: H_*(V/F) \to H_\lambda^*(C_c^\infty (\maG))$. It coincides with the image under $\chi$ of the Haefliger homology class associated with the basic cohomology class $\what{A} (\nu) \ch ({\what{E}})$. 
\end{theorem}

The proof relies on a long and tedious rescaling argument \`a la Getzler for transversely elliptic operators, see  \cite{BH18, BH19}. As a corollary, we get the following expected topological integrality result which opens up the way to more involved results when using the Connes-Moscovici residue cocycle for triangular structures and hence for general non-Riemannian foliations. Recall that $H$ is topologically generated by $h$ and that the fixed point submanifold $V^h=V^H$ is transverse to the foliation. 

\begin{corollary} \label{integrality}
The  characteristic number 
$$
\int_{V^h} \dfrac{\ch_\C (i^*[\sigma (E,d)](h))}{\ch_\C (\lambda_{-1}(N^h\otimes \C)(h))} \Td (TF^h\otimes \C) \what{A} (\nu^h) \ch ({\what{E}} \,|\,_{V^h}) \text{ belongs to $R(H)(h)$}.
$$
\end{corollary}

\begin{proof}
By the higher Lefschetz formula of \cite{BH10}, we have
$$
\int_{V^h} \dfrac{\ch_\C (i^*[\sigma (E,d)](h))}{\ch_\C (\lambda_{-1}(N^h\otimes \C)(h))} \Td (TF^h\otimes \C) \what{A} (\nu^h) \ch ({\what{E}} \,|\,_{V^h}) = \left\langle\left[\tau_{C_{\alpha}}\right] , \Ind^H(E,d) \right\rangle (h),
$$
 with $\alpha$ being the basic class $ \what{A} (R^\nu) \ch (\Ome^{\what{E}})$, $C_\alpha$ its corresponding Haefliger homology class, and $\left[\tau_{C_{\alpha}}\right]$ the associated $H$-equivariant cyclic cohomology class over $C_c^\infty (\maG)$ as defined in \cite{BH10}. By Theorem \ref{ECCc} the cyclic cohomology class of $\tau_{C_{\alpha}}$ coincides with the equivariant Connes-Chern character of $(\maA, \maH, \Dir^{\what{E}})$.
By Proposition \ref{EqP},  the pairing of the equivariant Connes-Chern character of $(\maA, \maH, \Dir^{\what{E}})$ with the equivariant $K$-theory takes values in $R(H)$. Hence, we deduce 
$$
\left\langle\left[\tau_{C_{\alpha}}\right] , \Ind^H(E,d)\right\rangle \; \in \; R(H),
$$
and so the conclusion.
\end{proof}

As explained above,  Theorem \ref{LefInt} is expected to still hold   for non-Riemannian foliations but using the more involved expression of the residue cocycle given in \cite{CM95}. However,  the characteristic expression of Corollary \ref{integrality} is specific to Riemannian foliations. As also explained above, one expects  similar integrality of characteristic numbers for non-Riemannian foliations, but  involving more exotic characteristic classes associated with actions of the Connes-Moscovici  Hopf algebra of vector fields.

\ms

We end this section by giving the heuristic explanation of the integrality result obatined in Corollary \ref{integrality} in the case where   $V^h$ is a strict transversal to $F$, say with dimension equal to $\codim (F)$. The characteristic number of Corollary \ref{integrality} then reduces to
$$
\int_{V^h} \dfrac{\ch_\C (i^*[\sigma (E,d)](h))}{\ch_\C (\lambda_{-1}(TF \, | \, _{V^h}\otimes \C)(h))} \what{A} (TV^h) \ch ({\what{E}} \,|\,_{V^h}),
$$
since, in this case, $N^h = TF \, | \, _{V^h}$,  $\nu^h = TV^h$,   $\Td (T\maF^h\otimes \C) = 1$ and $\what{A} (\nu^h)= \what{A} (TV^h)$. 
The $H$-invariant leafwise elliptic complex $(E,d)$ together with the transverse Dirac operator $\Dir^{\what{E}}$  allow one to define an $H$-invariant sharp product elliptic complex $(E,d) \sharp\Dir_{V}^{\what{E}}$ over the ambiant manifold $V$, which can be achieved for instance at the level of symbols. Thanks to the Atiyah-Segal Lefschetz formula \cite{AS3}, the above characteristic number is then equal to the Atiyah-Segal Lefschetz number of $h$ with respect to $(E,d) \sharp\Dir_{V}^{\what{E}}$, and hence belongs to $R(H)(h)$.

\end{document}